\documentclass[12pt]{amsart}

\usepackage{mathtools}
\usepackage{amssymb}
\usepackage{eucal}
\usepackage{mathrsfs}
\usepackage{xcolor}
\usepackage[percent]{overpic}
\usepackage{graphicx}
\usepackage{hyperref}
\usepackage{float}
\usepackage{tikz}

\newcommand{\auth}[1]{{\normalfont\scshape #1}}

\newcommand{\st}[1][T]{S_{{#1}}}
\newcommand{\omt}[2]{\Omega^{{#1}}({#2})}
\newcommand{\omT}[3][]{{}^{#1}\Omega^{{#2}}_{#3}}
\newcommand{\It}[2]{I^{{#1}}({#2})}
\newcommand{\IT}[2]{I^{{#1}}_{#2}}
\newcommand{\JT}[2]{J^{{#1}}_{#2}}
\newcommand{\kts}[1]{K^{#1}(t,s)}
\newcommand{\lts}[1]{L^{#1}(t,s)}

\newcommand{\br}{\bar\rho}
\newcommand{\bc}{\bar c}
\newcommand{\bk}{\bar k}
\newcommand{\eps}{\varepsilon}
\newcommand{\half}{\tfrac12}
\newcommand{\kz}{k_0}
\newcommand{\ko}{k_1}
\newcommand{\um}{\underline{M}(k)}
\newcommand{\uv}{\underline{v}}
\newcommand{\rr}{\mathbb{R}}

\newcommand{\sss}{\mathcal S}
\newcommand{\vv}{\CMcal V}
\newcommand{\linfty}{L^\infty}

\newcommand{\loneloc}{L_{\textnormal{loc}}^1}
\newcommand{\isp}[1]{\quad\mbox{#1}\quad}

\DeclareMathOperator{\supp}{supp}
\DeclareMathOperator{\esupp}{ess\, supp}
\DeclareMathOperator{\dist}{dist}

\newtheorem{lemma}{Lemma}
\newtheorem{theorem}{Theorem}
\newtheorem{corollary}{Corollary}

\numberwithin{equation}{section}

\theoremstyle{remark}
\newtheorem*{remark}{Remark}

\theoremstyle{definition}
\newtheorem*{definition}{Definition}

\title[Lifespan and Super-Classical Propagation]{
Lifespan Bounds and 
Super-Classical Propagation for 
Weak Solutions of the 
Three-Dimensional\\
 Compressible Euler Equations
}

\author{
Thomas C. Sideris
}

\address{
Department of Mathematics\\
University of California\\
Santa Barbara, CA 93106
}

\email{
sideris@ucsb.edu
}

\urladdr{
\url{
https://web.math.ucsb.edu/~sideris
}}

\begin{document}

\maketitle

\begin{abstract}
We establish an upper bound on the lifespan of bounded weak solutions
to the three-dimensional isentropic compressible Euler equations on
$[0,T)\times\rr^3$, under a localized positivity condition on the initial data and
the assumption that the essential support of the disturbance propagates into an
undisturbed background state no faster than predicted by classical local
well-posedness. We further show that any entropy-admissible bounded weak solution
with bounded inverse density that persists beyond this upper bound must propagate
into the undisturbed region at a strictly super-classical speed and that this
accelerated propagation is necessarily accompanied by a jump discontinuity in an
associated one-sided $L^\infty$-profile of the solution. 
\end{abstract}

\section[0]{Introduction}

This paper investigates upper bounds for the lifespan  for bounded weak solutions 
of the three-dimensional isentropic compressible Euler equations when the 
disturbance propagates no faster than the classical sound speed and describes the 
manner in which this propagation bound must fail if entropy-admissible 
solutions persist beyond those lifespan bounds.
A common feature of our arguments is the extraction of one-dimensional
information from the three-dimensional flow by means of planar averages
and moving half-spaces, without imposing any planar symmetry on the solution.

We work with the  equations of motion
for the fluid density $\rho>0$ and velocity $u\in\rr^3$ in nondimensional form  
\begin{subequations}
\begin{equation}
\label{CEE}
\begin{aligned}
&\partial_t\rho+\nabla\cdot(\rho u)=0\\
&\partial_t(\rho u)+\nabla\cdot(\rho u\otimes u)+\nabla p(\rho)=0,
\end{aligned}
\end{equation}
in which  the pressure $p$ is determined by the equation of state 
\begin{equation}
\label{press}
p(\rho)=\gamma^{-1}\rho^\gamma,\quad \gamma>1.
\end{equation}
\end{subequations}
Initial data $(\rho_0,u_0)\in L^\infty(\rr^3;\rr_{>0}\times\rr^3)$
are assigned at time $t=0$.
The precise notion of weak solution 
  on a strip $\st=[0,T)\times\rr^3$
 will be given in the next section.

Local well-posedness of finite energy perturbations  
\begin{equation*}
(\rho-\br,u)\in C([0,T);H^s(\rr^3)),
\end{equation*}
for any $s>5/2$ and  $\br>0$,
is a consequence of the  classical theory of positive symmetric hyperbolic systems, see \auth{Kato}
\cite{Kato-1975}.  Moreover,
 local uniqueness implies that if the initial data satisfies 
\begin{equation}
\label{introsupp}
\supp(\rho_0-\br,u_0)\subset \{x\in\rr^3:|x|\le1\},
\end{equation}
then the  disturbance propagates with finite speed in the sense that
\begin{equation}
\label{introsupp1}
\supp(\rho-\br,u)\subset\{(t,x)\in\st:|x|\le 1+\bc t\},
\end{equation}
 where $\bc$ is the sound speed
determined by the background density $\br$, namely $\bc=p'(\br)^{1/2}$.
In \cite{Sideris-1985}, \auth{Sideris} showed that the lifespan $T$ of these classical solutions
is   finite for
any  $H^s$ initial data satisfying \eqref{introsupp} and, for instance, under the compressive  positivity condition
\begin{equation}
\label{classicalcomp}
\rho_0(x)-\br>0,\quad x\cdot u_0(x)>0
\end{equation}
on a thin annulus
$\kz\le |x|<1$.  In particular, no assumption is made about the size of
$\|(\rho_0-\br,u_0)\|_{H^s}$.  This was later extended to two-dimensional flows
 by \auth{Rammaha} \cite{Rammaha-1989}.

The past two decades have brought substantial advances in the understanding 
of how multidimensional classical solutions evolve toward shock formation.
\auth{Christodoulou} \cite{Christodoulou-2007} developed a geometric theory of shock
 formation for the relativistic Euler equations with small compressive irrotational 
 initial perturbations of a constant equilibrium state, 
describing the geometry of the singular boundary of the maximal classical development 
and the mechanism responsible for shock formation. 
He also formulated the shock development problem for continuing the solution 
beyond this boundary and solved a restricted form of the problem in \cite{Christodoulou-2019}.
\auth{Christodoulou} and \auth{Miao}
\cite{Christodoulou-Miao-2014}  extended the geometric framework
to the nonrelativistic irrotational compressible Euler
equations.  
For nonzero vorticity, further major developments include the construction and
stability   of shock-forming solutions by
\auth{Luk} and \auth{Speck} \cite{Luk-Speck-2018,Luk-Speck-2024}
and by \auth{Buckmaster}, \auth{Shkoller}, and \auth{Vicol}
\cite{Buckmaster-Shkoller-Vicol-2023};
the analysis of the geometry of the maximal classical development and of shock formation by
\auth{Shkoller} and \auth{Vicol} \cite{Shkoller-Vicol-2024};
and, in the setting of azimuthal symmetry, the construction of post-shock entropy solutions by
\auth{Buckmaster}, \auth{Drivas}, \auth{Shkoller}, and \auth{Vicol}
\cite{Buckmaster-Drivas-Shkoller-Vicol-2022}.

In contrast, there is no general local well-posedness theory for weak solutions of systems of
hyperbolic conservation laws in multiple dimensions,
and an analogue of the maximal classical development is no longer available.
\auth{Rauch} \cite{Rauch-1986} showed that the BV class,
upon which the successful one-dimensional theory rests,
is unsuitable for local well-posedness.
In the absence of such a theory, we instead study bounded weak solutions,
a natural class that is sufficiently broad to accommodate shock formation.
However, at this level of generality, uniqueness is no longer guaranteed.
Building on convex-integration methods for the compressible Euler equations developed by
\auth{Chiodaroli}, \auth{De Lellis}, and \auth{Kreml}~\cite{CDLK-2015}, 
\auth{Chiodaroli}, \auth{Kreml}, \auth{M\'acha}, and 
\auth{Schwarzacher}~\cite{Chiodaroli-2021} demonstrated loss of uniqueness
within the entropy-admissible subclass for smooth initial data whose associated
classical solution breaks down.

Despite this lack of well-posedness, meaningful information can still be obtained
in multiple dimensions.
Our approach originates in the work of \auth{Sideris} \cite{Sideris-1984}, which
established upper bounds for the lifespan of weak solutions of certain nonlinear
hyperbolic systems under suitable largeness assumptions.
Our first main result, Theorem~\ref{thm1}, extends the finite lifespan theorem of
\auth{Sideris} \cite{Sideris-1985} for the compressible Euler equations
from the classical to the  bounded weak setting.
We consider a bounded weak solution $(\rho,u)$ of \eqref{CEE}, \eqref{press} on $\st$ such that
\begin{multline}
\label{wkboundsupp}
\esupp(\rho-1,u)\cap\{(t,x)\in\st:\kz+ t \le \omega\cdot x\}\\
\subset\{(t,x)\in\st:  |x|\le 1+ t\},
\end{multline}
for some  $0\le\kz<1$ and unit vector $\omega\in\rr^3$.
This localization reflects
 the classical  propagation  rate \eqref{introsupp1}, when $\br=\bc=1$.
 If the initial data satisfies a localized positivity condition analogous to
\eqref{classicalcomp}, then there is an explicit upper bound $T\le T_0$ depending only
on the parameter $\kz$, the index $\gamma$, and the initial data.
No restriction is placed on the size of the initial data.
In particular, if the amplitude of the data is $\mathcal O(\eps)$, for $\eps\ll1$, then the
upper bound for $T$  has the form $\exp(C/\eps)$ for some $C>0$, consistent with 
 the classical local existence theory for the irrotational initial value problem in three dimensions,
see \auth{Sideris}  \cite{Sideris-1991}.

 The method of proof measures the concentration of  the compressive disturbance
   ahead of the wave front $\omega\cdot x=s$
 via the quantity
 \begin{equation*}
v(t,s)=\int_{\{\omega\cdot x>s\}}(\rho(t,x)-1)(\omega\cdot x-s)^2dx,\quad s\ge\kz+t.
\end{equation*}
 Formally, the function $v(t,s)$ satisfies a  one-dimensional wave equation 
with unit propagation speed whose right-hand side is nonnegative.
 Instead of working directly with this equation, we recover  the d'Alembert representation formula
 for $v(t,s)$ from the weak formulation, using suitable  test functions.
 By the convexity of the equation of state in $\rho$ and the propagation hypothesis
 \eqref{wkboundsupp}, the representation yields an inhomogeneous nonlinear 
 integral inequality giving a lower bound for $v(t,s)$.
We then iterate this inequality, in the spirit of \auth{John}'s method in \cite{John-1979},
 to obtain an upper bound for the lifespan of weak solutions satisfying \eqref{wkboundsupp}.
 In this reduction, the effects of three-dimensional dynamics
are encoded in a time-dependent weighted volume of the
integration domain
\[
\{x\in\rr^3:s\le\omega\cdot x,\ |x|\le1+t\}.
\]

This leaves open the possibility that weak solutions  persist beyond $T_0$
  by propagating into the undisturbed background state at a super-classical rate. 
  Indeed, in the irrotational setting,
  \auth{Ginsberg} and \auth{Rodnianski} \cite{Ginsberg-Rodnianski-2024}
 construct global   three-dimensional weak
shock  solutions, close to Landau's two-shock $N$-wave profile, 
in which the exterior state ahead of the outgoing shock is the constant equilibrium state and the shock is logarithmically separated from the classical light cone.

  The remaining results address this possibility
  under the additional assumptions of entropy admissibility and boundedness of $\rho^{-1}$.
 Theorem~\ref{thm2} establishes a localized weak--strong uniqueness principle,
with the constant equilibrium state as the reference solution.
It is related to Theorem~5.2.1 of \auth{Dafermos} \cite{Dafermos-2016}
for general systems of hyperbolic conservation laws and Corollary~4.5 of
\auth{Wiedemann} \cite{Wiedemann-2018} for the Euler system, but the proof
given here does not invoke the standard Gronwall argument and yields
a quantitative estimate.
This estimate is applied in Corollary~\ref{cor1} to show
that the propagation speed of a perturbation into the undisturbed background state $(1,0)$
is bounded above by an explicit constant
   $c_0=1+C(\gamma)m$ where
   \[
  m= \|\rho^{-1}(\rho-1)\|_{L^\infty(\st)}+\|u\|_{L^\infty(\st)}.
   \]

The final main result, Theorem~\ref{thm3}, combines the preceding
uniqueness, finite-propagation, and lifespan results.
If an entropy-admissible weak solution on $\st$ persists beyond the
classical-speed lifespan bound $T_0$ furnished by
Theorem~\ref{thm1}, then there exists a transition time
$
0\le\tau\le T_0
$
which marks the onset of penetration into the undisturbed region
beyond the classical wave front.
 This penetration is discontinuous in the following sense. 
 For every $\tau<\tau'<T$,  there exist
a unit vector $\nu\in\rr^3$ and a speed $c_1$ with $1<c_1\le c_0$ such that,
upon setting
\[
{\mathcal H}^c=\{(t,x)\in(\tau,\tau')\times\rr^3:\nu\cdot x\ge1+\tau+c(t-\tau)\},
\]
 the nonincreasing function
\[
g(c)=\|\rho-1\|_{L^\infty\left({\mathcal H}^c\right)}
+\|u\|_{L^\infty\left({\mathcal H}^c\right)}, \quad 1\le c<\infty,
\]
has a jump discontinuity at $c=c_1$, with
$g(c_1^-)>g(c_1^+)=0$.
The consequences of this result for initial data with different
degrees of regularity at the wave front are examined in
Section~\ref{frontreg}.

 \section{Notation and Definitions}
\label{notation}
Given a fixed background density $\br=1$,  its associated sound speed is
$\bc=p'(\br)^{1/2}=1$.

For any $0<T<\infty$, we define the spacetime strip $\st=[0,T)\times\rr^3$.

\begin{definition}
\begin{subequations}
Given $T>0$,   we say that
a pair $(\rho,u)$ 
 is a  weak solution of the compressible Euler equations \eqref{CEE}, \eqref{press} on $\st$
with initial data $(\rho_0,u_0)$ if the following conditions hold:
\begin{equation}
\label{wkreg}
(\rho,u)\in \linfty(S_T;\rr_{>0}\times\rr^3),
\end{equation}
 \begin{equation}
\label{wkidreg}
(\rho_0,u_0)\in\linfty(\rr^3;\rr_{>0}\times\rr^3),
\end{equation}
\begin{equation}
\label{weakformulation1}
\iint_{S_T}(\rho\partial_t\varphi+\rho u\cdot \nabla_x\varphi)dxdt
+ \int_{\rr^3}\rho_0\varphi(0,\cdot)dx=0,
\end{equation}
and
\begin{equation}
\label{weakformulation2}
\begin{multlined}
\iint_{S_T}(\rho u\cdot\partial_t\psi+\rho\langle u\otimes u,\nabla_x\psi\rangle
+\tfrac1\gamma\rho^\gamma\nabla_x\cdot\psi)dxdt\\+\int_{\rr^3}\rho_0 u_0\cdot\psi(0,\cdot)dx=0,
\end{multlined}
\end{equation}
for all  pairs 
 \begin{equation*}
 (\varphi,\psi)\in
C^{0,1}_0(S_T;\rr\times\rr^3).
\end{equation*}
\end{subequations}
\end{definition}
That is, the test function space consists of compactly supported
 Lipschitz pairs $(\varphi,\psi)$
on the strip $S_T$.
We define weak solutions against Lipschitz test functions to facilitate the   proofs below. By standard mollification and dominated convergence, this is equivalent to the usual definition using $C_0^\infty$
  test functions.

Since $\rho$ and $\rho_0$ are defined only almost everywhere, the positivity assumptions in \eqref{wkreg} and \eqref{wkidreg} are understood in the almost-everywhere sense.

\begin{definition}
A weak solution  $(\rho,u)$ of \eqref{CEE}, \eqref{press} in $S_T$ with initial data
$(\rho_0,u_0)$ is said to be {\em entropy-admissible} 
if in addition
it satisfies the energy inequality
\begin{equation}
\label{admissible}
\begin{split}
\iint_{S_T}&\left[\left(\half\rho|u|^2+\tfrac{1}{\gamma(\gamma-1)}\rho^{\gamma}\right)\partial_t\varphi\right.\\
\phantom{S_T}
&+\left.\left(\half\rho|u|^2+\tfrac{1}{\gamma-1}\rho^{\gamma}\right)u\cdot\nabla_x\varphi\right]dxdt\\
&\qquad
+ \int_{\rr^3}\left(\half\rho_0|u_0|^2+\tfrac{1}{\gamma(\gamma-1)}\rho_0^{\gamma}\right)\varphi(0,\cdot) dx
\ge0,
\end{split}
\end{equation}
for all  nonnegative scalar test functions 
\begin{equation*}
\varphi\in C^{0,1}_0(S_T;\rr_{\ge0}).
\end{equation*}
\end{definition}

Every bounded classical solution in $ C^1(S_T;\rr_{>0}\times\rr^3)$ is an entropy-admissible 
 weak solution
 because it satisfies the energy equality.

  If $(\rho,u)$ is an
entropy-admissible weak solution on $S_T$, then
since $(\rho,u)\in L^\infty(S_T)\subset \loneloc(S_T)$, for almost all
$\sigma\in(0,T)$, there exists
\[
        (\rho_\sigma,u_\sigma)\in
        L^\infty(\rr^3;\rr_{>0}\times\rr^3)
\]
such that, for every compact set $K\subset\rr^3$,
\[
 \lim_{h\downarrow0}
 \frac1{2h}\int_{\sigma-h}^{\sigma+h}
 \|(\rho(t,\cdot),u(t,\cdot))-(\rho_\sigma,u_\sigma)\|_{L^1(K)}\,dt
 =0.
\]
For every such $\sigma$, a standard time-cutoff argument shows that
the restriction of $(\rho,u)$ to
$[\sigma,T)\times\rr^3$ is an entropy-admissible weak solution on the
shifted strip, with initial data $(\rho_\sigma,u_\sigma)$.

We shall refer to $\sigma$ as a Lebesgue time for the map
$t\in[0,T)\mapsto (\rho(t,\cdot),u(t,\cdot))\in \loneloc(\rr^3)$,
and the pair $(\rho_\sigma,u_\sigma)$ as a Lebesgue value.

\bigskip

To describe the propagation of the essential support of weak disturbances, we introduce the following notation.

For a fixed unit vector $\omega\in\rr^3$, domain parameters $0\le\kz<\ko$, 
and time $t\ge0$, define the
half-space
\begin{equation*}
\omt{\kz,\infty}{t}=\{x\in\rr^3: \kz+  t\le \omega\cdot x\}
\end{equation*}
and the spherical cap
\begin{equation*}
\omt{\kz,\ko}{t}
=\omt{\kz,\infty}t\cap\{x\in\rr^3:|x|\le \ko+t\}.
\end{equation*}
Given $0<T<\infty$,  we define the spacetime region
\begin{equation*}
\omT{\kz,\ko}T=\left\{(t,x)\in \st:x\in\omt{\kz,\ko}{t}\right\}.
\end{equation*}
The region $\omT{\kz,\infty}{T}$ represents a spacetime half-space intersected with $S_T$,
and $\omT{\kz,\ko}{T}$, $0\le\kz<\ko<\infty$,  describes  the portion of the forward  light  cone  
contained in $\omT{\kz,\infty}{T}$. 

\section{Main Results}

\begin{theorem}[Finite lifespan under classical-speed propagation]
\label{thm1}

 Let $(\rho,u)$ be a  weak solution 
 of \eqref{CEE}, \eqref{press}
on $S_T$, where $0<T< \infty$, with initial
 data $(\rho_0,u_0)$, so that
   \eqref{wkreg}, \eqref{wkidreg}, \eqref{weakformulation1},
 and \eqref{weakformulation2}
 are in force.

For fixed unit vector $\omega\in \rr^3$ and parameter $0\le\kz<1$, let $\omt{\kz,1}0$
and $\omT{\kz,1}T$ 
be defined as above.

Define the function
\begin{subequations}
 \begin{equation}
\label{wzdef}
 w_0(k)=\half\int\limits_{\omt{k,1}{0}}
\left({\omega\cdot\rho_0 u_0(x)}+{\rho_0(x)}-1\right)
(\omega\cdot x-k)^2dx.
\end{equation}
 Under \eqref{wkidreg}, $w_0$ is continuous on $\kz\le k\le1$.
\medskip
   
Suppose that within $\omT{\kz,\infty}T$
the disturbance  propagates into the constant state $(\br,0)=(1,0)\in\rr_{>0}\times\rr^3$ with speed at most $\bc=1$:
   \begin{equation}
\label{esupp}
\esupp(\rho-1,u)\cap \omT{\kz,\infty}T
\subset \omT{\kz,1}T,
\end{equation}
 and  that at the front
    the initial disturbance is weakly outgoing:
\begin{equation}
\label{wkdatacond}
\text{ $w_0$
is positive and decreasing for $\kz\le k<1$}.
\end{equation}
\end{subequations}

Then  the lifespan $T$ has the upper bound
\begin{subequations}
\begin{equation}
\label{concl}
 T\le T_0(\kz), 
\end{equation}
where 
\begin{equation}
\label{etub0}
T_0(\kz)
=
\inf_{\kz<k<1}T(k,\kz)-\frac{1+\kz}{2}.
\end{equation}
and
\begin{equation}
\label{etub}
T(k,\kz)
=
\begin{cases}
\displaystyle
\exp\left[
\Gamma(\gamma)
\frac{(1-\kz)^6}{(k-\kz)w_0(k)}
\right],
& \gamma\ge2,
\\[4mm]
\displaystyle
\exp\left[
\begin{aligned}
&\Gamma(\gamma)
\frac{(1-\kz)^6}{(k-\kz)w_0(k)}
\\[-1mm]
&\quad\times
\left(
1+\frac{3w_0(k)}{(1-k)^4}
\right)^{2-\gamma}
\end{aligned}
\right],
& 1<\gamma<2.
\end{cases}
\end{equation}
The constant
$\Gamma(\gamma)$ depends only  on $\gamma$.

\end{subequations}
\end{theorem}

 \begin{remark}
The positivity of the density a.e.\ has been assumed explicitly.
Even for smooth solutions, the restriction of the initial data to
$\omt{\kz,\infty}0$ does not determine the solution throughout
$\omT{\kz,\infty}T$, since information may enter through the
lateral boundary.
\end{remark}

\begin{remark}
The hypothesis \eqref{wkdatacond} on the initial data
can be replaced by the assumption that $w_0(k)>0$, $\kz\le k<1$.
In this case,  $w_0(k)$ would be replaced  by
$\underline{w}_0(k)=\min\{w_0(s):\kz\le s\le k\}$  in  $T(k,\kz )$, see \eqref{remark}.
\end{remark}

\begin{remark}
We shall see in the proof of Theorem~\ref{thm1} that the support property
\eqref{esupp} for the solution implies the corresponding localized support
property for the data, namely \eqref{esupp0.1} and \eqref{esupp0.2}.
\end{remark}

\begin{theorem}[Localized weak--strong uniqueness near the background]
\label{thm2}
 Suppose   that $(\rho,u)$ is an entropy-admissible weak solution on $S_T$ 
 with initial data $(\rho_0,u_0)$, and
assume
 that $\|\rho^{-1}\|_{L^\infty(S_T)}<\infty$. Let $m>0$ satisfy
\begin{equation}
\label{ellinfbound}
\|\rho^{-1}(\rho-1)\|_{L^\infty(S_T)}+\|u\|_{L^\infty(S_T)}\le m.
\end{equation}

Let $\vv \subset\rr^3$ be a nonempty,  closed, and bounded set with Lipschitz boundary.
For $0<\alpha\le1$ define the function
\begin{equation*}
d_\alpha(t,x)=(\dist(x,\rr^3\setminus\vv)-t/\alpha)_+,\quad (t,x)\in\st,
\end{equation*}
and then define the associated family of spacetime sets 
\begin{equation*}
\vv_\alpha=\supp d_\alpha\cap\st.
\end{equation*}
Define the bounded, nondecreasing, left-continuous function 
\begin{equation*}
F(\alpha;\vv)=\|\rho^{-1}(\rho-1)\|_{L^\infty\left(\vv_\alpha\right)}
+\|u\|_{L^\infty\left(\vv_\alpha\right)},
\quad  0<\alpha\le1,
\end{equation*}
and set
\begin{equation*}
F(0;\vv)=\|\rho_0^{-1}(\rho_0-1)\|_{L^\infty(\vv)}+\|u_0\|_{L^\infty(\vv)}.
\end{equation*}

If the initial data is such that $F(0;\vv)=0$, then 
\begin{equation*}
\alpha(\vv)=\sup\left\{0\le\alpha\le1: F(\alpha;\vv)=0\right\}\ge(1+C(\gamma) m )^{-1}, 
\end{equation*}
where the constant $C(\gamma)>0$ depends only on $\gamma$.

Furthermore, if $\alpha(\vv)<1$, then 
\begin{equation*}
F(\alpha;\vv)\ge F(\alpha(\vv)^+;\vv)>0, \quad \alpha(\vv)< \alpha\le1.
\end{equation*}

Finally, under the assumption $F(0;\vv)=0$, if for some
$0<\alpha\le1$ the weak solution $(\rho,u)$ admits a
continuous representative on $\vv_\alpha$, then the disturbance
vanishes a.e.\ in $\vv_\alpha$.

\end{theorem}

\begin{corollary}[Finite propagation into the background]

\label{cor1}

Suppose that $(\rho,u)$ is an entropy-admissible weak solution on $S_T$
with initial data $(\rho_0,u_0)$, assume that
$\|\rho^{-1}\|_{L^\infty(S_T)}<\infty$,  let $m>0$ satisfy
\eqref{ellinfbound}, and
let $c_0= 1+C(\gamma)m$,
with $C(\gamma)$ as in Theorem~\ref{thm2}.

If for some $R>0$ the initial disturbance satisfies
 \begin{equation*}
\esupp(\rho_0-1,u_0)
\subset\{x\in\rr^3:|x|\le R\},
\end{equation*}
then the disturbance satisfies
\begin{equation*}
\esupp(\rho-1,u)
\subset\{(t,x)\in\st:|x|\le R+c_0 t\}.
\end{equation*}

More generally, let $0<\sigma<T$ be a Lebesgue time for
$(\rho,u)$, and let $(\rho_\sigma,u_\sigma)$ denote the corresponding
Lebesgue value.  If for some $R>0$
\[
        \esupp(\rho_\sigma-1,u_\sigma)
        \subset \{x\in\rr^3: |x|\le R\},
\]
then
\begin{multline*}
\esupp(\rho-1,u)\cap\bigl((\sigma,T)\times\rr^3\bigr)\\
\subset
\left\{
(t,x)\in(\sigma,T)\times\rr^3:
|x|\le R+c_0(t-\sigma)
\right\}.
\end{multline*}

\end{corollary}

\begin{theorem}[Super-classical penetration and jump profile]
\label{thm3}

Suppose that $(\rho,u)$ is an entropy-admissible weak solution on $S_T$
with initial data $(\rho_0,u_0)$, assume that
$\|\rho^{-1}\|_{L^\infty(S_T)}<\infty$,  let $m>0$ satisfy
\eqref{ellinfbound}, and
let $c_0= 1+C(\gamma)m$,
with $C(\gamma)$ as in Theorem~\ref{thm2}.

Assume that the initial disturbance satisfies 
\begin{equation}
\label{esupp0.0}
\esupp(\rho_0-1,u_0)
\subset \{x\in\rr^3:|x|\le1\}.
\end{equation}
Fix a unit vector $\omega\in\rr^3$ and $0\le\kz<1$, and let
 $w_0$ be defined by \eqref{wzdef}.  Assume that $w_0$ satisfies \eqref{wkdatacond}.

Assume that the lifespan satisfies $T>T_0(\kz)$, where $T_0(\kz)$ was
defined in \eqref{etub0}, \eqref{etub}.  
Set
\[
\mathcal T
=
\left\{
0<t<T:
\esupp(\rho-1,u)\cap
\left\{(s,x)\in S_t:|x|>1+s\right\}
\ne\emptyset
\right\}.
\]
Then
$\mathcal T\ne\emptyset$, and, setting
\[
 \tau=\inf\mathcal T,
 \]
we have\[
0\le\tau\le T_0(\kz).
\]

Moreover, there holds
\begin{subequations}
\begin{equation}
\label{nospread}
\esupp(\rho-1,u)\cap\st[\tau]\subset\{(t,x)\in \st[\tau]:|x|\le1+t\},
\end{equation}
and for every $\tau<\tau'<T$,
\begin{multline}
\label{exteriorsupp}
\esupp(\rho-1,u)\cap\left((\tau,\tau')\times\rr^3\right)\\
\subset \biggl\{(t,x)\in(\tau,\tau')\times\rr^3:
|x|\le 1+\tau+c_0(t-\tau)\biggr\}
\end{multline}
while
\begin{multline}
\label{spread0}
\esupp(\rho-1,u)
\cap\biggl\{(t,x)\in(\tau,\tau')\times\rr^3:\\
1+t<|x|\le 1+\tau+c_0(t-\tau)\biggr\}\ne\emptyset.
\end{multline}
\end{subequations}

For every $\tau<\tau'<T$, there exist a unit vector
$\nu\in\rr^3$ and a speed $c_1$ satisfying
\[
1<c_1\le c_0
\]
such that, upon setting
\begin{equation}
\label{jumpplane}
\mathcal H_{\tau,\tau'}^{\nu,c}
=
\left\{
(t,x)\in (\tau,\tau')\times\rr^3:
\nu\cdot x\ge1+\tau+c(t-\tau)
\right\},
\end{equation}
for $c\ge1$, the nonincreasing profile
\[
g(c)
=
\|\rho-1\|_{L^\infty\left(\mathcal H_{\tau,\tau'}^{\nu,c}\right)}
+
\|u\|_{L^\infty\left(\mathcal H_{\tau,\tau'}^{\nu,c}\right)},
\qquad 1\le c<\infty,
\]
has a jump discontinuity at $c=c_1$, with
\[
g(c_1^-)>g(c_1^+)=0.
\]

\end{theorem}

\begin{remark}
The jump discontinuity in Theorem~\ref{thm3} is consistent with shock
formation, although the theorem does not by itself identify a shock
hypersurface.  
\end{remark}

\section{Front Regularity and the Onset of Super-Classical Propagation}

\label{frontreg}

Fix a unit vector $\omega\in\rr^3$ and a parameter $0\le\kz<1$, and assume that the standing hypotheses of Theorem~\ref{thm3} hold, apart
from the condition $T>T_0(\kz)$.
Then these hypotheses hold for $\kz$ replaced by any $\lambda$ satisfying $\kz\le \lambda<1$.
For $T_0$, defined in \eqref{etub0}, \eqref{etub}, set
\[
T^\ast=\liminf_{\lambda\uparrow1}T_0(\lambda).
\]
Assume that $T^\ast<T$. Choose a sequence $\{\lambda_n\}_{n\in\mathbb N}$ in $(\kz,1)$ such that
\[
\lambda_n\uparrow1,\quad
T_0(\lambda_n)\to T^\ast, \isp{and }
T_0(\lambda_n)<T.
\]
 Theorem~\ref{thm3} applies  with $\kz$ replaced by $\lambda_n$ for each $n$, and gives
 \[
 \tau\le T_0(\lambda_n).
 \]
Since the transition time $\tau$ is defined solely in terms of the solution, it
 is independent of $n$.
Letting $n\to\infty$, we obtain
\[
0\le\tau\le T^\ast.
\]
In particular, if $T^\ast=0$, then $\tau=0$, signaling instantaneous 
penetration beyond the classical unit cone.

We now give a sufficient condition on the behavior of the initial data
near the front which ensures $T^\ast =0$. 
To isolate this effect, assume that
\[
(\rho_0,u_0)\in C^{j,\mu}(\Omega^{\kz,\infty}(0)),
\qquad
j\in\mathbb Z_{\ge0},\quad \mu\in(0,1],
\]
and that, for some $0<\varepsilon\ll1$,
\[
\omega\cdot\rho_0u_0(x)+\rho_0(x)-1
\sim
\varepsilon(1-|x|^2)^{j+\mu},
\qquad x\in\Omega^{\kz,1}(0),
\]
while
\[
\omega\cdot\rho_0u_0(x)+\rho_0(x)-1=0,
\qquad
x\in\Omega^{\kz,\infty}(0)\setminus\Omega^{\kz,1}(0).
\]
Here $\sim$ denotes two-sided comparability, with constants
independent of $\varepsilon$.

By the definition of $w_0$ in \eqref{wzdef}, this assumption implies
that $w_0(k)>0$ for $\kz \le k<1$, and
\[
w_0(k)\sim\varepsilon(1-k)^{4+j+\mu},
\qquad \kz\le k<1,
\]
with constants independent of $k$ and $\varepsilon$.
The positivity of the integrand also gives the monotonicity
assumption in \eqref{wkdatacond}.

Choose $\kz\le\lambda<1$ and let $k=(1+\lambda)/2$. The preceding estimate for $w_0$,
together with \eqref{etub0}, \eqref{etub}, gives
\begin{equation}
\label{tzex}
\begin{aligned}
T_0(\lambda)
&\le T((1+\lambda)/2,\lambda)-(1+\lambda)/2  \\
&\le
\exp\left(
        \Gamma'(1-\lambda)^{1-j-\mu}/\varepsilon
        \right)
        -(1+\lambda)/2,
\end{aligned}
\end{equation}
where $\Gamma'$ is independent of $\lambda$ and
$\varepsilon$.

If the data is merely H\"older continuous at the front, namely
$(\rho_0,u_0)\in C^{0,\mu}$ with $0<\mu<1$, then $0<j+\mu<1$.
It follows from \eqref{tzex} that
\[
        \lim_{\lambda\uparrow1} T_0(\lambda )=T^\ast=0.
\]
Thus $\tau=0$, and the loss of classical finite-speed confinement
occurs immediately. For every $0<\tau'<T$, Theorem~\ref{thm3}
produces a one-sided profile having a jump discontinuity at some
speed $1<c_1\le c_0$.

This is consistent with the elementary Burgers mechanism: entropy
solutions with $C^{0,1/2}$ initial data such as
\[
u_0(x)=
\begin{cases}
1, & x\le-1,\\
(-x)^{1/2}, & -1<x<0,\\
0, & x\ge0.
\end{cases}
\]
develop a shock immediately.  The analogy is only qualitative, but it
illustrates the same role played by sub-Lipschitz regularity at the
front.  Although we do not consider vacuum background states
$\bar\rho=0$ here, the explicit affine solutions in \cite{Sideris-2017}, for
which $w_0(k)\sim (1-k)^{4+\delta}$ with $0<\delta<1$, exhibit the
same  vanishing-rate mechanism behind instantaneous
spreading.

If the data is Lipschitz at the front, then $j+\mu=1$.
From \eqref{tzex}, we obtain
\[
        T^\ast\le \exp(\Gamma'\varepsilon^{-1})-1<\infty.
\]
Under the standing assumption $T>T^\ast$, we therefore have
\[
0\le\tau\le T^\ast,
\]
and for every $\tau<\tau'<T$, the jump-profile conclusion of
Theorem~\ref{thm3} holds.

For more regular data, with $j+\mu>1$, the bound \eqref{tzex} degenerates
as $\lambda\uparrow1$.  Thus the present method gives no finite upper
bound for $T^\ast$ when the initial outgoing functional vanishes faster
than the Lipschitz rate at the front.
This should not be interpreted as evidence that classical propagation
persists for smoother data; rather, the localized functional used here
becomes too small near the front to yield a uniform upper bound as the
half-space approaches the initial wave front.
This conclusion concerns only the limiting procedure $\lambda\uparrow1$.
For each fixed $\kz<1$ satisfying the hypotheses of Theorem~\ref{thm3},
persistence beyond $T_0(\kz)$ still forces the super-classical penetration
and jump profile described there.

\section{Proof of Theorem~\ref{thm1}}

\subsection*{Preliminaries}
We are given a  weak solution
$(\rho,u)$ on $S_T$ with initial data $(\rho_0,u_0)$.
The goal is to establish the  upper bound \eqref{concl} for $T$, the lifespan of the solution.  

As in \auth{Sideris} \cite{Sideris-1985}, the proof is direct: we derive a priori inequalities for an arbitrary solution satisfying the hypotheses and then read off the resulting restriction on $T$.

Following the strategy of \auth{Sideris} \cite{Sideris-1985},
the proof proceeds directly, by deriving a priori inequalities for an
arbitrary solution satisfying the hypotheses and then reading off the
resulting restriction on $T$.
Set $\bar k=1-\kz$. Since $T_0(\kz)\ge \bar k/2$, the desired
bound is immediate if $T<\bar k/2$. 
It remains to consider the case $T\ge \bar k/2$.
The system is rotationally invariant, so we may assume without loss of generality that
$\omega = e_1$ and write $x_1=e_1\cdot x$.  

An important role in the proofs of both 
Theorems \ref{thm1} and \ref{thm2} will be played by the function
 \begin{equation}
\label{Phi1}
\begin{aligned}
\Phi(\eta) =\ &p(1+\eta)-p(1)-p'(1)\eta\\
\ =\ &
\tfrac1\gamma(1+\eta)^\gamma-\tfrac1\gamma-\eta,
\end{aligned}
\end{equation} 
defined on the interval $1+\eta>0$.\footnote{In the special case $\gamma=2$, we have $\Phi(\eta)=\eta^2/2$.}
We note that $\Phi(0)=\Phi'(0)=0$ and $\Phi''(\eta)>0$. 
 Thus, crucially, $\Phi$ is strictly convex and nonnegative on
$(-1,\infty)$, and increasing on $[0,\infty)$.

 Having chosen $\omega=e_1$, we  identify $\rr^3$ with $\rr\times\rr^2$.
 It is then convenient to have a notation for the projections of the domains $\omt{\kz,\ko}t$
and $\omT{\kz,\ko}T$ onto the $(t,x_1)$-plane.
Thus, we define the intervals
\begin{equation*}
\It{\kz,\ko}{t}=\{s\in \rr : \kz+ t\le s\le \ko+ t\},
\end{equation*}
\begin{equation*}
\It{\kz,\infty}{t}=\{s\in \rr : \kz+ t\le s\},
\end{equation*}
and the slab
\begin{equation*}
\IT{\kz,\ko}{T}=\{(t,s)\in[0,T)\times \rr:s\in \It{\kz,\ko}{t}\}.
\end{equation*}

To  motivate the first portion of the proof, suppose temporarily that the solution belongs to
$C^1(\omT{\kz,\infty}T)$.  A direct computation then shows that the function
\begin{subequations}
\begin{equation}
\label{vdef}
v(t,s)=\int_{\{x_1>s\}}(\rho(t,x)-1)(x_1-s)^2dx
\end{equation}
belongs to $C^2(\IT{\kz,\infty}T)$ and satisfies
\begin{equation*}
(\partial_t^2-\partial_s^2)v(t,s)
=2G(t,s),
\end{equation*}
in $\IT{\kz,\infty}T$, with
\begin{equation}
\label{Gdef}
G(t,s)=\int_{\{x_1> s\}}
\big[\rho (e_1\cdot u)^2(t,x)+\Phi(\rho(t,x)-1)\big]dx.
\end{equation}
\end{subequations}
D'Alembert's formula can then be used to provide a representation for $v(t,s)$.
For weak solutions, this computation is only formal, 
but the resulting representation formula remains meaningful.
We now derive it directly from the weak formulation, yielding \eqref{wkrep}.

We retain the definitions of $v$ and $G$ in the weak setting.
To justify the manipulations below, we explicitly record some basic facts.

 \begin{lemma}{\rm (Basic Integrability and Confinement Properties)}
\label{lem1}
Suppose that $(\rho, u)\in L^\infty(\st;\rr_{>0}\times\rr^3)$ 
satisfies the localized propagation speed bound \eqref{esupp}. 
Let $v(t,s)$ and $G(t,s)$ denote the half-space integrals defined as in 
\eqref{vdef} and \eqref{Gdef}, respectively.
 Then, 
\begin{enumerate}
    \item\label{item1} \textup{(Spacetime Integrability and Support)} 
    The fluid variables satisfy $\rho-1, u \in L^1(\Omega^{\kz ,\infty}_T)$.
    The half-space integrals  satisfy $v, G \in L^1(I^{\kz ,\infty}_{T})$, and their spacetime essential supports 
    are confined such that 
    \begin{equation*}
\esupp (v,G) \cap I^{\kz,\infty}_{T} \subset I^{\kz,1}_{T}.
\end{equation*}
    \item\label{item2} \textup{(Time-Slice Properties)} For almost every fixed $t \in (0, T)$, 
    the spatial profiles satisfy 
    \[
    \rho(t,\cdot)-1, u(t,\cdot) \in L^1(\Omega^{\kz ,\infty}(t))
    \]
    and
    \[
    v(t,\cdot), G(t,\cdot) \in L^1(I^{\kz ,\infty}(t)).
    \]
     Furthermore, their spatial essential supports are confined such that:
    \begin{gather*}
    \esupp (\rho(t,\cdot)-1,u(t,\cdot))\cap\omt{\kz,\infty}t\subset\omt{\kz,1}t,
\intertext{and}
\esupp (v(t,\cdot),G(t,\cdot))\cap \It{\kz,\infty}t\subset\It{\kz,1}t.
\end{gather*}
        
    \item \label{item3} \textup{($L^1$-Norms)} For any $0 < T' \le  T$ and any parameter $k$ such that $\kz  \le k \le 1$,
    $v$ satisfies the integrability estimate
        \begin{gather*}
\|v\|_{L^1\left(\IT{k,\infty}{T'}\right)}
\le \tfrac13\iint_{\Omega^{k,1}_{T'}} |\rho-1|(x_1-t-k)^3 dx dt,
\intertext{while $G$ satisfies the identity}
\|G\|_{L^1\left(\IT{k,\infty}{T'}\right)}
=\iint_{\Omega^{k,1}_{T'}} \left[ \rho(u \cdot \omega)^2 + \Phi(\rho-1) \right](x_1-t-k) dx dt.
\end{gather*}
\end{enumerate}
\end{lemma}

\begin{proof}
The localized support condition \eqref{esupp} implies that the functions
$\rho-1$ and $u$ vanish a.e.\ in
$\omT{\kz,\infty}T\setminus \omT{\kz,1}T$.
Since $(\rho,u)\in L^\infty(\st)$ and $\omT{\kz,1}T$ has finite measure,
we have
\[
        \rho-1,\;u\in L^1(\omT{\kz,\infty}T).
\]
Also, $\Phi$ extends continuously to $[-1,\infty)$, and
$\rho$ and $u$ are bounded. Hence
\[
\rho(e_1\cdot u)^2,\;\Phi(\rho-1)
\in L^1(\omT{\kz,\infty}T).
\]
Fubini's theorem therefore implies that the half-space integrals defining
$v$ and $G$ belong to $L^1(\IT{\kz,\infty}T)$ and yields the estimate and
identity in \eqref{item3}.  Taking $k=1$ in \eqref{item3} gives
$v=G=0$ a.e.\ on $\IT{1,\infty}T$, which is equivalent to the stated
spacetime support confinement.

The corresponding time-slice assertions follow from the preceding spacetime
integrability and vanishing statements by another application of Fubini.
\end{proof}

\subsection*{Reduction to a one-dimensional representation formula}

We first show that the localization hypothesis on the solution implies the same localization for the initial data.
That is, we show that the hypothesis \eqref{esupp} of Theorem~\ref{thm1}
implies the corresponding localized support properties \eqref{esupp0.1} and \eqref{esupp0.2} for the initial data.

By \eqref{weakformulation1} and \eqref{weakformulation2},
 for any pair of test functions $(\varphi,\psi)\in
 C^{0,1}_0(S_T;\rr\times\rr^3)$, there holds the modified weak formulation
\begin{subequations}
\begin{equation}
\label{weakformulation3}
\begin{multlined}
\iint_{\st}[(\rho-1)\partial_\tau\varphi+\rho u\cdot \nabla_x\varphi]dxd\tau\\
+ \int_{\rr^3}(\rho_0-1)\varphi(0,\cdot)dx=0,
\end{multlined}
\end{equation}
and 
\begin{equation}
\label{weakformulation4}
\begin{multlined}
\iint_{\st}[\rho u\cdot\partial_\tau\psi+\rho\langle u\otimes u,\nabla_x\psi\rangle
+\Phi(\rho-1)\nabla_x\cdot\psi]dxd\tau\\
+\iint_{\st}(\rho-1)\nabla_x\cdot\psi dxd\tau
+\int_{\rr^3}\rho_0 u_0\cdot\psi(0,\cdot)dx=0.
\end{multlined}
\end{equation}
\end{subequations}

If
$\supp(\varphi,\psi)\subset \omT{\kz,\infty}T\setminus \omT{\kz,1}T$,
then by \eqref{esupp}, $\supp(\varphi,\psi)$ is disjoint from $\esupp(\rho-1,u)$, 
and so by \eqref{weakformulation3}   we see that
\begin{equation*}
\int_{\rr^3}(\rho_0-1)\varphi(0,\cdot)dx=0.
\end{equation*}
Since the values of compactly supported Lipschitz functions may be chosen arbitrarily inside
$\omT{\kz,\infty}T\setminus \omT{\kz,1}T$, the vanishing of the initial integral implies the initial support inclusion
\begin{subequations}
\begin{equation}
\label{esupp0.1}
\esupp(\rho_0-1)\cap \omt{\kz,\infty}0\subset \omt{\kz,1}0.
\end{equation}

Likewise,  by \eqref{weakformulation4}, \eqref{esupp}, and \eqref{esupp0.1},
we have
\begin{equation*}
\int_{\rr^3}u_0\cdot\psi(0,\cdot)dx=
\int_{\rr^3}\rho_0u_0\cdot\psi(0,\cdot)dx=0,
\end{equation*}
so that
\begin{equation}
\label{esupp0.2}
\esupp u_0\cap \omt{\kz,\infty}0\subset \omt{\kz,1}0.
\end{equation}
\end{subequations}

The support properties \eqref{esupp0.1} and \eqref{esupp0.2}
show that $w_0(1)=0$. We extend $w_0$ continuously to
$[\kz,\infty)$ by setting
\[
w_0(k)=0,\qquad k\ge1.
\]

By \eqref{esupp}, if the pair $(\varphi,\psi)$ satisfies the compatibility condition
\begin{subequations}
\begin{equation}
\label{compatibility1}
\supp(\varphi,\psi)\subset\omT{\kz,\infty}T\isp{and} \partial_\tau\psi+\nabla_x\varphi=0
\isp{on}
\omT{\kz,1}T,
\end{equation}
then adding \eqref{weakformulation3} and \eqref{weakformulation4}, we obtain
\begin{equation}
\begin{multlined}
\label{compatibility2.1}
\iint_{\omT{\kz,\infty}T}(\rho-1)(\partial_\tau\varphi+\nabla_x\cdot\psi) dxd\tau\\
+\iint_{\omT{\kz,\infty}T}[\rho\langle u\otimes u,\nabla_x\psi\rangle
+\Phi(\rho-1)
\nabla_x\cdot\psi]dxd\tau\\
+ \int_{\omt{\kz,\infty}0}(\rho_0-1)\varphi(0,\cdot)dx\\+\int_{\omt{\kz,\infty}0}\rho_0 u_0\cdot\psi(0,\cdot)dx=0.
\end{multlined}
\end{equation}
\end{subequations}

We now construct a family of test functions 
 $(\varphi,\psi)$, 
 consistent with the compatibility condition \eqref{compatibility1},
 that will lead to our desired representation.

For parameter values $(t,s)\in\IT{\kz,\infty}{T}$, we begin by defining
\begin{subequations}
\begin{equation}
\label{thetadef}
\theta(\tau,x_1;t,s)=\tfrac16
\left[(x_1-\tau-s+t)_+^3
-(x_1+\tau-s-t)_+^3\right].
\end{equation}
Then $\theta(\,\cdot\,;t,s)\in C^{2,1}(\rr^2)$ and 
\begin{equation}
\label{lwave}
(-\partial_\tau^2+\partial_{x_1}^2)\theta(\,\cdot\,;t,s)=0.
\end{equation}
This particular solution of the one-dimensional
wave equation is chosen because its second spatial derivative 
produces the kernels appearing in the desired representation formula.
\end{subequations}

To localize the construction in time and space, we introduce two cutoff functions.
First, let $h\in C^\infty(\rr)$ satisfy
\begin{subequations}
\begin{equation}
\label{hdef}
h(\sigma)=
\begin{cases}
1,& 1\le \sigma<\infty\\
0,&-\infty<\sigma\le0
\end{cases}
\isp{and}
h'\ge0.
\end{equation}
Note that
\begin{equation}
\label{approxid}
\int_0^1h'(\sigma )d\sigma =1 \isp{and }\int_0^1\sigma  h''(\sigma )d\sigma =-1.
\end{equation} 
\end{subequations}
This function will serve as a cut-off in time.

\begin{subequations}
Next,   choose a  spatial cut-off  $\zeta\in C^\infty_0(\rr^3)$ such that
  \begin{equation}
\label{zetasupp}
\zeta(x)=1,\isp{for all} x\in\bigcup_{0\le t<T}\omt{\kz,1}t.
\end{equation}
Thanks to the support properties \eqref{esupp}, \eqref{esupp0.1},  \eqref{esupp0.2}, and the properties
of $\zeta$ above, 
we see that
\begin{equation}
\label{nocutoff}
\zeta\;(\rho-1, u)=(\rho-1,u) \isp{a.e.\ } \omT{\kz,\infty}T,
\end{equation}
and
\begin{equation}
\label{nocutoff0}
\zeta\;(\rho_0-1,u_0)=(\rho_0-1,u_0 )\isp{a.e.\ } \omt{\kz,\infty}0.
\end{equation}
\end{subequations}

Finally, setting
\begin{equation*}
\varphi(\tau,x;t,s)=-\partial_\tau \left[h\left((t-\tau)/\eps\right)\theta(\tau,x_1;t,s)\right]\zeta(x)
\end{equation*}
{and}
\begin{equation*}
\psi(\tau,x;t,s)=\nabla \left[h\left((t-\tau)/\eps\right)\theta(\tau,x_1;t,s)\right]\zeta(x),
\end{equation*}
we have $(\varphi,\psi)(\,\cdot\,;t,s)\in C^{0,1}_0(\st;\rr\times\rr^3)$.
Since $s-t\ge\kz$, we have
\[
\supp(\varphi,\psi)\subset\omT{\kz,\infty}T.
\]
Observe  that by \eqref{zetasupp}, the formulas reduce to
\begin{equation*}
\varphi(\tau,x;t,s)=-\partial_\tau \left[h\left((t-\tau)/\eps\right)\theta(\tau,x_1;t,s)\right]
\end{equation*}
and
\begin{equation*}
\psi(\tau,x;t,s)= h\left((t-\tau)/\eps\right)\partial_{x_1}\theta(\tau,x_1;t,s)e_1,
\end{equation*}
on $\omT{\kz,1}T$.
Thus, the compatibility condition \eqref{compatibility1} is satisfied, and
\begin{multline*}
\partial_\tau\varphi(\tau,x;t,s)+\nabla\cdot\psi(\tau,x;t,s)\\
=(-\partial_\tau^2+\partial_{x_1}^2)[ h\left((t-\tau)/\eps\right)\theta(\tau,x_1;t,s)].
\end{multline*}

In view of  \eqref{nocutoff} and \eqref{nocutoff0},  substitution
of $(\varphi,\psi)$ into \eqref{compatibility2.1} produces an equation of the form
\begin{equation*}
V(t,s,\eps)=W_1(t,s,\eps)+W_2(t,s,\eps),
\end{equation*}
with
\begin{equation*}
V(t,s,\eps)
=\iint\limits_{\omT{\kz,\infty}T}(\rho-1)(\partial_\tau^2-\partial_{x_1}^2)
[ h\left((t-\tau)/\eps\right)\theta(\tau,x_1;t,s)] dxd\tau,
\end{equation*}
\begin{multline*}
W_1(t,s,\eps)=-\int_{\omt{\kz,\infty}0}(\rho_0-1)\partial_\tau \left[h\left((t-\tau)/\eps\right)\theta(\tau,x_1;t,s)\right]\big|_{\tau=0}dx\\
+\int_{\omt{\kz,\infty}0}(e_1\cdot \rho_0 u_0) h\left((t-\tau)/\eps\right)\partial_{x_1}\theta(\tau,x_1;t,s)\big|_{\tau=0}dx,
\end{multline*}
and 
\begin{multline*}
W_2(t,s,\eps)\\
=\iint\limits_{\omT{\kz,\infty}T}[\rho(e_1\cdot u)^2+\Phi(\rho-1)]
h\left((t-\tau)/\eps\right)
\partial_{x_1}^2\theta(\tau,x_1;t,s) dxd\tau.
\end{multline*}
Fix $(t,s)\in\IT{\kz,\infty}T$ and take $0<\eps<t$.
We evaluate the three terms separately before passing to the limit
$\eps\to0$.

Consider the expression $W_1$ first.  
Since $t/\eps>1$, we have $h(t/\eps)=1$ and $h'(t/\eps)=0$.   So using the definition of $\theta$
in \eqref{thetadef}, we obtain 
\begin{align*}
W_1(t,s,\eps)\ =\
&
\int\limits_{\omt{\kz,\infty}0}(\rho_0-1)\half 
[(x_1-s+t)_+^2+(x_1-s-t)_+^2]\; dx\\
&
+\int\limits_{\omt{\kz,\infty}0} (e_1\cdot\rho_0 u_0)\half 
[(x_1-s+t)_+^2-(x_1-s-t)_+^2]\; dx\\
\ =\ &W_1(t,s,0),
\end{align*}
for all $0<\eps<t$.
Upon rearrangement, we get
\begin{equation*}
W_1(t,s,0)=v_0(t,s),
\end{equation*}
with
\begin{equation*}
\begin{aligned}
v_0(t,s)\ = \
&
\half\int\limits_{\omt{\kz,\infty}0}[(\rho_0-1)+e_1\cdot\rho_0 u_0](x_1-s+t)_+^2\; dx\\
&+\half\int\limits_{\omt{\kz,\infty}0}[(\rho_0-1)-e_1\cdot\rho_0 u_0](x_1-s-t)_+^2\; dx\\
\ = \
&
\half\int\limits_{\{x_1>s-t\}}[(\rho_0-1)+e_1\cdot\rho_0 u_0](x_1-s+t)^2\; dx\\
&+\half\int\limits_{\{x_1>s+t\}}[(\rho_0-1)-e_1\cdot\rho_0 u_0](x_1-s-t)^2\; dx.
\end{aligned}
\end{equation*}
By the support properties \eqref{esupp0.1} and \eqref{esupp0.2},
$v_0$  is continuous on $\IT{\kz,\infty}T$, 
and $\supp v_0\cap\IT{\kz,\infty}T\subset
\IT{\kz,1}T$.  Notice that 
\begin{equation}
\label{wzsimplification}
v_0(t,s)=w_0(s-t),\isp{when} s+t\ge1,
\end{equation}
where $w_0$ is the function defined in \eqref{wzdef}.

Next, we turn to $W_2$.  From \eqref{hdef},
we see that the  integration domain is contained in
$\omT{\kz,\infty}t$.
By the dominated convergence theorem, we have
\begin{align*}
\lim_{\eps\to0} W_2(t,s,\eps)
\ = \ & 
\iint_{\omT{\kz,\infty}t}
[\rho(e_1\cdot u)^2+\Phi(\rho-1)]\partial_{x_1}^2\theta(\,\cdot\,;t,s) dxd\tau\\
\ =\ & W_2(t,s,0).
\end{align*}
Now by the definition of $\theta$ in  \eqref{thetadef},
\begin{equation*}
\partial_{x_1}^2\theta(\tau,x_1;t,s)=(x_1-\tau-s+t)_+-(x_1+\tau-s-t)_+.
\end{equation*}
Thus, we are led to consider integrals of the form
\begin{equation*}
\iint_{\omT{\kz,\infty}t}
[\rho(e_1\cdot u)^2+\Phi(\rho-1)](x_1-\tau-k)_+ dxd\tau,
\end{equation*}
with $k\ge \kz$.
Lemma~\ref{lem1}, \eqref{item3} implies that
\begin{equation*}
\begin{aligned}
\iint_{\omT{\kz,\infty}t}&
[\rho(e_1\cdot u)^2+\Phi(\rho-1)](x_1-\tau-k)_+\; dxd\tau\\
&=\int_0^t\int_{k+\tau}^\infty G(\tau,\xi)d\xi d\tau,
\end{aligned}
\end{equation*}
for $\kz\le k\le1$, but in fact, this holds for all $\kz\le k$, since both sides vanish for $k\ge1$,
by Lemma  \ref{lem1}.

Applying the same Fubini calculation to the lower limits
$s-t+\tau$ and $s+t-\tau$, we obtain
\begin{align*}
W_2(t,s,0)\ =\ &\int_0^t\int_{s-t+\tau}^\infty G(\tau,\xi)d\xi d\tau-
\int_0^t\int_{s+t-\tau}^\infty G(\tau,\xi)d\xi d\tau\\
\ =\ &\iint_{K(t,s)}G(\tau,\xi)d\xi d\tau,
\end{align*}
where
\begin{equation*}
K(t,s)=\{(\tau,\xi)\in\rr^2:|\xi-s|\le t-\tau,\; 0\le \tau\le t\}
\end{equation*}
is the usual one-dimensional backward characteristic cone.
Note that $K(t,s)\subset\IT{\kz,\infty}T$, whenever $(t,s)\in \IT{\kz,\infty}T$.
Since $G\in L^1(\IT{\kz,\infty}T)$, by Lemma~\ref{lem1},
we conclude that $W_2(t,s,0)$ is continuous on $\IT{\kz,\infty}T$, 
and $\supp W_2(\cdot,\cdot,0)\cap \IT{\kz,\infty}T\subset \IT{\kz,1}T$.
\allowdisplaybreaks

It remains to evaluate  the term $V$.  With the aid of \eqref{lwave}, we have
\begin{equation*}
V(t,s,\eps)=V_1(t,s,\eps)+V_2(t,s,\eps),
\end{equation*}
with
\begin{equation*}
V_1(t,s,\eps)\ 
=-\iint\limits_{\omT{\kz,\infty}T}(\rho-1)\tfrac2\eps h'\left((t-\tau)/\eps\right)\partial_\tau\theta(\,\cdot\,;t,s)dxd\tau
\end{equation*}
and
\begin{equation*}
V_2(t,s,\eps)\ =\ \iint\limits_{\omT{\kz,\infty}T}
(\rho-1)
\tfrac1{\eps^2} h''\left((t-\tau)/\eps\right)\theta(\,\cdot\,;t,s)dxd\tau.
\end{equation*}

We fix $t_0$ with $0<t_0<T$, and suppose $0<\eps<t_0$.  Let $(t,s)\in\IT{\kz,\infty}T$ with $t\ge t_0$.
To shorten the formulas, set $\tilde\rho=\rho-1$.
Upon performing the change of variables $\sigma =(t-\tau)/\eps$
and taking into account \eqref{thetadef} and \eqref{hdef},
we can write 
\begin{align*}
V_1(t,s,\eps)\ =\ &\int_0^1\int_{\{x_1>s-\eps\sigma \}}\tilde\rho(t-\eps\sigma ,x)
h'(\sigma )(x_1-s+\eps\sigma )^2  dxd\sigma \\
&+\int_0^1\int_{\{x_1>s+\eps\sigma \}}\tilde\rho(t-\eps\sigma ,x)
h'(\sigma )(x_1-s-\eps\sigma )^2 dxd\sigma.
\end{align*}
This can be re-arranged in the form
\begin{align*}
V_1(t,s,&\eps)\\
\ =\ &
\int_0^1\int_{\{x_1>s\}}\tilde\rho(t-\eps\sigma ,x)h'(\sigma )(x_1-s+\eps\sigma )^2
 dxd\sigma \\
&
+\int_0^1\int_{\{x_1>s\}}\tilde\rho(t-\eps\sigma ,x)
h'(\sigma )(x_1-s-\eps\sigma )^2
 dxd\sigma \\
&+\int_0^1\int_{\{-\eps\sigma <x_1-s<0\}}\tilde\rho(t-\eps\sigma ,x)
h'(\sigma )(x_1-s+\eps\sigma )^2 dxd\sigma \\
&-\int_0^1\int_{\{0<x_1-s<\eps\sigma \}}\tilde\rho(t-\eps\sigma ,x)h'(\sigma )(x_1-s-\eps\sigma )^2
 dxd\sigma.
 \end{align*}
After cancellations, this becomes
\begin{align*}
V_1(t,s,\eps)
 \ =\ & \int_0^1\int_{\{x_1>s\}}\tilde\rho(t-\eps\sigma ,x)2h'(\sigma )(x_1-s )^2
 dxd\sigma\\
 & +\mathcal{O}\left(\eps^2\|\tilde\rho\|_{L^\infty\left(\omT{\kz,1}{T}\right)}\right).
\end{align*}
In the same way, we have
\begin{align*}
V_2(t,s,\eps)
\ =\ & \int_0^1\int_{\{x_1>s\}}\tilde\rho(t-\eps\sigma ,x)\sigma h''(\sigma )(x_1-s )^2
 dxd\sigma\\
 & +\mathcal{O}\left(\eps^2\|\tilde\rho\|_{L^\infty\left(\omT{\kz,1}{T}\right)}\right).
\end{align*}
Thus, we find that 
\begin{multline*}
V(t,s,\eps)\\=\int_0^1\int_{\{x_1>s\}}\tilde\rho(t-\eps\sigma ,x)
[2h'(\sigma )+\sigma  h''(\sigma )](x_1-s)^2 dxd\sigma\\
+\mathcal{O}\left(\eps^2\|\tilde\rho\|_{L^\infty\left(\omT{\kz,1}{T}\right)}\right).
\end{multline*}

We claim that $V(t,s,\eps)\to v(t,s)$ in $L^1\left(\IT{\kz,\infty}T\setminus \IT{\kz,\infty}{t_0}\right)$,
for every $t_0$.
Recall the properties
$v\in L^1\left(\IT{\kz,\infty}{T}\right)$ 
and $\esupp v\cap \IT{\kz,\infty}T\subset\IT{\kz,1}T$ given in Lemma~\ref{lem1}, \eqref{item1}.
Using the properties of $h$ in \eqref{approxid}, we have $\int_0^1[2h'(\sigma )+\sigma  h''(\sigma )]d\sigma =1$.
Thus, we obtain
\begin{align*}
\|&V(\cdot,\eps)-v(\cdot)\|_{L^1\left(\IT{\kz,\infty}T\setminus \IT{\kz,\infty}{t_0}\right)}\\
&=\int_{t_0}^{T}\int_{\kz+t}^{\infty}\big|V(t,s,\eps)-v(t,s)\big|dsdt\\
&=\int_{t_0}^{T}\int_{\kz+t}^{\infty}\bigg|\int_0^1\int_{\{x_1>s\}}
[\tilde\rho(t-\eps\sigma ,x)-\tilde\rho(t,x)]
 \\
&\qquad
\times
[2h'(\sigma )+\sigma  h''(\sigma )](x_1-s)^2
 dxd\sigma \bigg|dsdt+\mathcal O(\eps^2)\\
&\lesssim \int_0^1\int_{t_0}^{T}\int_{\{x_1>\kz+t\}}
|\tilde\rho(t-\eps\sigma ,x)-\tilde\rho(t,x)|
 dxdtd\sigma+\mathcal O(\eps^2).
\end{align*}
Since $0<\eps<t_0$, the shifted times $t-\eps\sigma$ remain in $(0,T)$.
Setting $r=\eps\sigma$, we obtain
\begin{multline*}
\|V(\cdot,\eps)-v(\cdot)\|_{L^1\left(\IT{\kz,\infty}T\setminus \IT{\kz,\infty}{t_0}\right)}\\
\lesssim\sup_{r\le\eps}\int_{t_0}^{T}\int_{\{x_1>\kz+t\}}
|\tilde\rho(t-r ,x)-\tilde\rho(t,x)|dxdt+\mathcal O(\eps^2).
\end{multline*}
Set
\[
q(t,x)=\mathbf 1_{\{x_1>\kz+t\}}\tilde\rho(t,x),
\]
and extend $q$ to be zero outside $\st$.  Then $q\in L^1(\rr\times\rr^3)$, by Lemma~\ref{lem1}.
Since
\[
\mathbf 1_{\{x_1>\kz+t\}}[q(t-r,x)-q(t,x)]
=\mathbf 1_{\{x_1>\kz+t\}}[\tilde\rho(t-r,x)-\tilde\rho(t,x)]
\]
 for $0<r<\eps$, we have that the last integral above is bounded by
\[
\sup_{r<\eps}\|q(\cdot-r,\cdot)-q(\cdot,\cdot)\|_{L^1(\rr\times\rr^3)}.
\]
This tends to zero as $\eps\to0$ by the $L^1$-continuity
of time translations,
which therefore proves the claim.

Combining the preceding computations and passing to the limit $\eps\to0$, 
we arrive at the desired representation formula
\begin{equation}
\label{wkrep}
v(t,s)=v_0(t,s)+\iint_{\kts{}}G(\tau,\xi)d\xi d\tau \isp{a.e.} \IT{\kz,\infty}{T}.
\end{equation}
The right-hand side is continuous on $\IT{\kz,\infty}{T}$, so we can redefine 
$v(t,s)$ on a set of measure zero, if necessary, so that $v(t,s)\in C\left(\IT{\kz,\infty}T\right)$
and equality holds in \eqref{wkrep}
for all $(t,s)\in \IT{\kz,\infty}{T}$.  
Moreover, we can now strengthen the inclusions in Lemma~\ref{lem1}
by replacing $\esupp v$ with $\supp v$:
\begin{equation*}
\supp v\cap\IT{\kz,\infty}T\subset\IT{\kz,1}T,
\quad
 \supp v(t,\cdot)\cap\It{\kz,\infty}t\subset \It{\kz,1}t,\; t\in[0,T).
\end{equation*}
That $v$ has a continuous representative
 is not unexpected, cf.\ Lemma 1.3.3 in \auth{Dafermos} \cite{Dafermos-2016}.

Lemma~\ref{lem1}, \eqref{item1}, gives
$\esupp G\cap\IT{\kz,\infty}T\subset  \IT{\kz,1}T$.
Since $(t,s)\in\IT{\kz,\infty}T$, the backward characteristic cone $K(t,s)$ is
contained in $\IT{\kz,\infty}T$.
Because $G=0$ a.e.\ outside $\IT{\kz,1}T$, the integration in \eqref{wkrep} is effectively
restricted to the truncated backward characteristic cone
\begin{equation*}
\kts{\kz,1}=K(t,s)\cap \IT{\kz,1}T.
\end{equation*}
This set is nonempty only for $(t,s)\in\IT{\kz,1}T$;
see Figure \ref{diagram0}.

\begin{figure}[h]
\caption{The truncated backward characteristic cone
$\kts{\kz,1}=K(t,s)\cap \IT{\kz,1}T$ associated with a point
$(t,s)\in \IT{\kz,1}T$.}
\label{diagram0}

\centering
\vspace{0.25cm}

\begin{tikzpicture}[x=0.92cm,y=0.92cm]
  \def\yT{5.20}
  \def\xLzero{0.70}
  \def\xRzero{4.10}
  \def\xaxisend{10.25}
  \def\yaxisend{6.20}

  \coordinate (L0) at (\xLzero,0);
  \coordinate (LT) at ({\xLzero+\yT},\yT);
  \coordinate (R0) at (\xRzero,0);
  \coordinate (RT) at ({\xRzero+\yT},\yT);

  \coordinate (P) at (4.90,3.20);

\fill (P) circle (1.5pt);

  \coordinate (A) at ({4.90-3.20},0);
  \coordinate (B) at ({4.90+3.20},0);

  \coordinate (Q) at
    ({(4.90+3.20+\xRzero)/2},
     {(4.90+3.20-\xRzero)/2});

  \draw
    (L0) -- (LT)
    node[
      pos=.69,
      sloped,
      above=4pt
    ]
    {\small{$s-t=\kz$}};

  \draw
    (R0) -- (RT)
    node[
      pos=.77,
      sloped,
      below=4pt
    ]
    {\small{$s-t=1$}};

  \draw
    (LT) -- (RT)
    node[
      pos=.56,
      above=4pt
    ]
    {\small{$t=T$}};

  \draw
    (P) -- (B);

  \draw[line width=1.15pt]
    (A) -- (P) -- (Q) -- (R0);

  \draw[line width=1.15pt]
    (A) -- (R0);

    \node[
  anchor=north west,
  xshift=-12pt,
  yshift=18pt
] at (P)
  {\small{$(t,s)$}};


\node[rotate=47] at (4.40,1.45)
  {$\kts{\kz,1}$};

  \node at (6.68,3.88)
    {$\IT{\kz,1}T$};

  \node at (7.95,2.02)
    {$v=0$};

  \node at (5.93,0.82)
    {$v=0$};

  \draw[->,line width=.85pt]
    (0,0) -- (0,\yaxisend)
    node[above right=-1pt] {$t$};

  \draw[->,line width=.85pt]
    (0,0) -- (\xaxisend,0)
    node[above left=-1pt] {$s$};
\end{tikzpicture}
\end{figure}
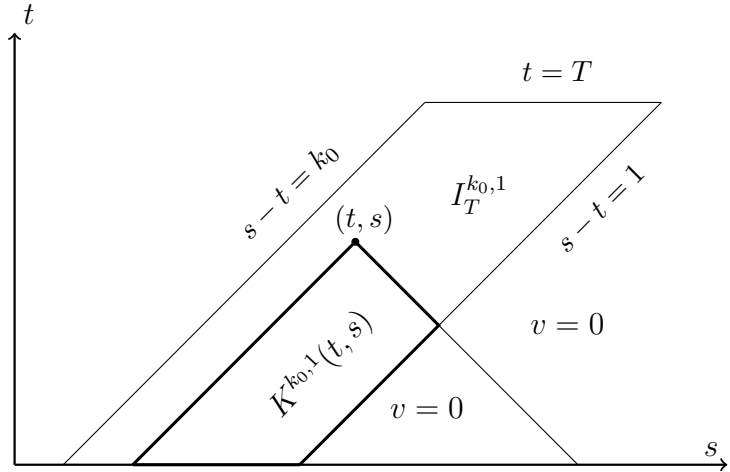


\subsection*{Integral inequality for $v$}
The next step will be to use \eqref{wkrep} to obtain a nonlinear inhomogeneous integral inequality involving $v$.

By definition, if $(t,s)\in\IT{\kz,\infty}T$, then $\{x\in\rr^3:x_1\ge s\}\subset \omt{\kz,\infty}t$.
Lemma~\ref{lem1}, \eqref{item2}, then tells us that
\begin{align*}
\{x\in\rr^3:x_1\ge s\}\cap &\esupp(\rho(t,\cdot)-1,u(t,\cdot))\\
& \ \subset\ 
\{x\in\rr^3:x_1\ge s\}\cap \omt{\kz,1}t\\
& \ =\ \{x\in\rr^3:x_1\ge s,\; |x|\le 1+t\}\\
& \ =\ \omt{s-t,1}t,
\end{align*}
a.e.\ $0<t<T$.
Hence, the domains of integration defining $v(t,s)$ and
$G(t,s)$ may both be reduced to $\omt{s-t,1}t$, for almost
all $(t,s)\in\IT{\kz,1}T$.
The admittedly opaque notation $\omt{s-t,1}t$ will be used only in this section. 

We  introduce the important quantity
\begin{equation}
\label{mass}
M(t,s)
=
\int\limits_{\omt{s-t,1}t}(x_1-s)^2dx,\quad (t,s)\in \IT{\kz,1}T,
\end{equation}
which encodes the three-dimensional dependence in the problem.
This function will be computed explicitly below.  For now,
note that  whenever $M(t,s)>0$
\begin{equation*}
d\mu=\frac{(x_1-s)^2}{M(t,s)}dx
\end{equation*}
is a probability measure on $\omt{s-t,1}t$.
This normalization will allow us to exploit the convexity of $\Phi$.

Given any $(t,s)\in  \IT{\kz,1}T$ 
and $x\in \omt{s-t,1}t$, 
it follows that 
\begin{equation*}
0\le x_1-s\le |x|-s\le 1 -s+t\le 1 -\kz= \bk.
\end{equation*}
 Together with the nonnegativity of $\Phi$, 
this yields that, a.e.\ $\IT{\kz,1}T$,
\begin{equation*}
\begin{aligned}
G(t,s)\ =\ & \int\limits_{\omt{s-t,1}t}[\rho(e_1\cdot u)^2+\Phi(\rho(t,x)-1)]dx\\
\ \ge\ & \int\limits_{\omt{s-t,1}t}\Phi(\rho(t,x)-1)dx\\
\ \ge\ & \bk^{-2} \int\limits_{\omt{s-t,1}t}\Phi(\rho(t,x)-1)(x_1-s)^2dx\\
\ =\ &\bk^{-2}M(t,s) \int\limits_{\omt{s-t,1}t}\Phi(\rho(t,x)-1)d\mu. 
\end{aligned}
\end{equation*}
As noted in \eqref{Phi1}, $\Phi(\eta)$ is convex on the interval $(-1,\infty)$.
Since $\rho(t,x)-1>-1$ a.e.\ and both $\rho-1$ and $\Phi(\rho-1)$ belong to $L^1(\mu)$, 
Jensen's inequality provides the lower bound
\begin{equation}
\label{nonlb}
\begin{aligned}
G(t,s)\ \ge\ & \bk^{-2}M(t,s)\Phi\left(\int_{\omt{s-t,1}t}(\rho(t,x)-1)d\mu 
\right)\\
\ =\ &\bk^{-2}M(t,s)\Phi\left(\frac{v(t,s)}{M(t,s)}\right),
\end{aligned}
\end{equation}
for almost all $(t,s)\in  \IT{\kz,1}T$.

Combining \eqref{wkrep} and  \eqref{nonlb}, we arrive at
the inequality
\begin{equation}
\label{rep}
v(t,s)\ge
v_0(t,s)
+\iint\limits_{\kts{\kz,1}}\bk^{-2}M(\tau,\xi)\Phi\left(\frac{v(\tau,\xi)}{M(\tau,\xi)}\right)d\xi d\tau,
\end{equation}
for all $(t,s)\in\IT{\kz,1}T$, since \eqref{wkrep} holds
everywhere and \eqref{nonlb} holds almost everywhere on the
integration region.

We have therefore reduced the problem to the scalar nonlinear  inhomogeneous 
integral inequality \eqref{rep}.
The remainder of the proof is devoted to showing that if $v(t,s)$ is a continuous
function on $\IT{\kz,1}T$ satisfying \eqref{rep}, then $T$ has the upper bound \eqref{concl} claimed in the theorem.

\subsection*{Formulation of the iterative problem}
In order to simplify the estimation of $v_0(t,s)$, it is helpful 
to have $t+s\ge1 $, so that \eqref{wzsimplification} holds.  
At the same time, we would like to ensure that $s-t$ is bounded uniformly away from 
one  so that  both $v_0(t,s)$ and $M(t,s)$ remain uniformly bounded away from zero.  
Accordingly, we choose a number $k$ such that $\kz<k<1 $
and further restrict ourselves to the  subregion
\begin{equation*}
\JT{\kz,k}T=\IT{\kz,k}T\cap \{(t,s):s+t\ge1\}\subset\IT{\kz,1}T.
\end{equation*}
Henceforth the iterative argument will be carried out on $\JT{\kz,k}T$.
Since the integrand in the representation formula \eqref{rep}
is nonnegative, only the portion of the backward characteristic cone lying inside $\JT{\kz,k}T$   will be used.
The region   $\JT{\kz,k}T$ is nonempty for any value of the parameter $\kz<k<1$ thanks to our assumption that
\[
T\ge\bk/2=(1 -\kz)/2>(1-k)/2.
\]
 This leads us to  define 
\begin{equation*}
\lts{\kz,k}=
K(t,s)\cap \JT{\kz,k}T\subset\kts{\kz,k}.
\end{equation*}
See Figure \ref{diagram}.

 For the time being, we shall regard $k$ as being fixed. 
However, its dependence will always be indicated explicitly so that we may later optimize over $k$.

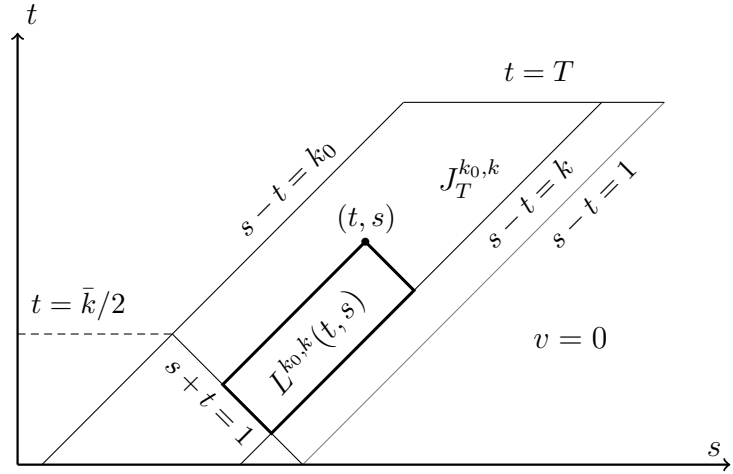
\begin{figure}[h]
\caption{The iteration region $\JT{\kz,k}T$, depicted with $T>\bk/2$,
and at a representative point $(t,s)\in\JT{\kz,k}T$, the truncated
backward characteristic cone
$\lts{\kz,k}=K(t,s)\cap\JT{\kz,k}T$.}
\label{diagram}

\centering
\vspace{0.25cm}

\begin{tikzpicture}[x=0.92cm,y=0.92cm]
  \def\yT{5.20}
  \def\xKzero{0.35}   
  \def\xK{3.20}
  \def\xOne{4.10}     
  \def\xSum{4.10}     
  \def\xaxisend{10.25}
  \def\yaxisend{6.20}

  \coordinate (L0) at (\xKzero,0);
  \coordinate (LT) at ({\xKzero+\yT},\yT);

  \coordinate (K0) at (\xK,0);
  \coordinate (KT) at ({\xK+\yT},\yT);

  \coordinate (O0) at (\xOne,0);
  \coordinate (OT) at ({\xOne+\yT},\yT);

  \coordinate (JL) at
    ({(\xSum+\xKzero)/2},{(\xSum-\xKzero)/2});

  \coordinate (JR) at
    ({(\xSum+\xK)/2},{(\xSum-\xK)/2});

  \coordinate (S0) at (\xSum,0);

  \coordinate (P) at (5.00,3.20);
  
\fill (P) circle (1.5pt);

  \pgfmathsetmacro{\pminus}{5.00-3.20}
  \pgfmathsetmacro{\pplus}{5.00+3.20}

  \coordinate (A) at
    ({(\xSum+\pminus)/2},{(\xSum-\pminus)/2});

  \coordinate (Q) at
    ({(\pplus+\xK)/2},{(\pplus-\xK)/2});

  \draw
    (L0) -- (LT)
    node[
      pos=.72,
      sloped,
      above=0pt
    ]
    {\small{$s-t=\kz$}};

  \draw
    (K0) -- (KT)
    node[
      pos=.76,
      sloped,
      below=-1pt
    ]
    {\small{$s-t=k$}};

  \draw[gray]
    (O0) -- (OT)
    node[
      pos=.76,
      sloped,
      below=0pt,
      text=black
    ]
    {\small{$s-t=1$}};

  \draw
    (LT) -- (OT)
    node[
      pos=.52,
      above=4pt
    ]
    {\small{$t=T$}};

  \draw
    (S0) -- (JL)
    node[
      pos=.6,
      sloped,
      below=0pt
    ]
    {\small{$s+t=1$}};

  \draw[densely dashed]
    (0,{(\xSum-\xKzero)/2}) -- (JL);

  \node[above right=1pt]
    at (0,{(\xSum-\xKzero)/2})
    {\small{$t=\bk/2$}};

  \draw[line width=1.15pt]
    (A) -- (P) -- (Q) -- (JR) -- cycle;

    
    \node[
  font=\small,
  anchor=south west,
  xshift=-15pt,
  yshift=-1pt
] at (P)
  {$(t,s)$};

  \node[rotate=49]
    at (4.30,1.75)
    {$\lts{\kz,k}$};

  \node
    at (6.50,4.10)
    {$\JT{\kz,k}T$};

  \node
    at (7.95,1.82)
    {$v=0$};

  \draw[->,line width=.85pt]
    (0,0) -- (0,\yaxisend)
    node[above right=-1pt] {$t$};

  \draw[->,line width=.85pt]
    (0,0) -- (\xaxisend,0)
    node[above left=-1pt] {$s$};
\end{tikzpicture}
\end{figure}

Combining \eqref{wzsimplification} and \eqref{wkdatacond} gives
the lower bound\footnote{If $w_0$ 
were merely nonnegative for $\kz\le k<1$, then we could use $\underline{w}_0(k)=\min\{w_0(s):\kz\le s\le k\}$
here instead of $w_0(k)$.}
\begin{equation}
\label{remark}
v_0(t,s)=w_0(s-t)\ge w_0(k)>0,\quad (t,s)\in\JT{\kz,k}T.
\end{equation}

By \eqref{remark}, inequality \eqref{rep} becomes
\begin{equation}
\label{rep2}
v(t,s)
\ge
w_0(k)
+\iint\limits_{\lts{\kz,k}}\bk^{-2}M(\tau,\xi)\Phi\left(\frac{v(\tau,\xi)}{M(\tau,\xi)}\right)d\xi d\tau,
\end{equation}
provided $(t,s)\in\JT{\kz,k}T$.
Since $\Phi\ge0$, the integral term in
\eqref{rep2} is nonnegative, and therefore we have  established the positivity of $v$:
\begin{equation*}
v(t,s)\ge w_0(k)>0,\quad (t,s)\in\JT{\kz,k}T.
\end{equation*}
This uniform lower bound is the starting point for the iteration. 
In particular, every subsequent lower estimate for $v$ will dominate the positive constant
$w_0(k)$.

Since $\Phi$ is increasing on $[0,\infty)$, \eqref{rep2} immediately yields the following more general estimate.
 Whenever $\uv(t,s)$ satisfies
\begin{equation}
\label{genlb}
v(t,s)\ge \uv(t,s)\ge w_0(k)>0,\quad (t,s)\in \JT{\kz,k}T,
\end{equation}
then
\begin{equation}
\label{rep3}
v(t,s)-w_0(k)
\ge 
\iint\limits_{\lts{\kz,k}}\bk^{-2}
M(\tau,\xi )\Phi\left(\frac{\uv(\tau,\xi )}{M(\tau,\xi )}\right)d\xi d\tau,
\end{equation}
on $\JT{\kz,k}T$.

Before carrying out the iteration, we first obtain quantitative  bounds for
the function $\Phi$ and 
 the quantity $M(t,s)$.  
 These estimates will be used to control the nonlinear term in \eqref{rep2}.

By the definition \eqref{Phi1}, we know that
 $\Phi(0)=\Phi'(0)=0$, and so Taylor's theorem implies that when $\eta>0$
\begin{equation*}
\Phi(\eta)=\half\eta^2\Phi''(\eta^\ast)=\tfrac{\gamma-1}2\eta^2(1+\eta^\ast)^{\gamma-2},
\end{equation*}
for some $0\le\eta^\ast\le\eta$.
Consequently, if $\eta\ge 0$, then
\begin{equation}
\label{Phi2}
\Phi(\eta)\ge \eta^2\Psi(\eta)\isp{with}
\Psi(\eta)=
\begin{cases}
\frac{\gamma-1}2(1+\eta)^{\gamma-2},& 1<\gamma\le2,\\
\frac{\gamma-1}2, & 2<\gamma.
\end{cases}
\end{equation}

Identifying $x\in\rr^3$ with $(y,z)\in \rr\times\rr^2$, 
we may  write
\begin{equation*}
\omt{s-t,1}t= \{(y,z)\in\rr\times\rr^2: (y^2+|z|^2)^{1/2}\le 1 +t,\; y\ge s\}.
\end{equation*}
Then by the definition \eqref{mass} 
\begin{equation*}
M(t,s)=\int_s^{1+t}\int_{(y^2+|z|^2)^{1/2}\le 1+t}(y-s)^2dzdy.
\end{equation*}
Evaluating the integral explicitly gives
\begin{equation*}
M(t,s)=
\tfrac\pi{30} (4(1 +t)+s)(1 -s+t)^4,
\end{equation*}
for all $(t,s)\in \IT{\kz,1}T$.  
This is the only place in the proof where the dimension is used.

\begin{subequations}

Hence, on $\IT{\kz,k}T$, we find that 
\begin{equation}
\label{massest}
M(t,s)
\le
 \tfrac{4\pi}{30} (1 +s+ t)(1 -s+t)^4
\le
 \bk^4(1 +s+ t),
\end{equation}
and
\begin{equation}
\label{MLB}
M(t,s)
\ge \tfrac{4\pi}{30} (1 -s+t)^4
\ \ge\
\tfrac1{3}(1-k)^4\equiv \um.
\end{equation}

\end{subequations}

Returning to \eqref{rep3}, the lower bound \eqref{Phi2} yields 
\begin{equation}
\label{rep3.1}
v(t,s)-w_0(k)
\ge 
\bk^{-2}\iint\limits_{\lts{\kz,k}}
\frac{\uv(\tau,\xi )^2}{M(\tau,\xi )}\Psi\left(\frac{\uv(\tau,\xi )}{M(\tau,\xi )}\right)d\xi d\tau.
\end{equation}
Employing   \eqref{massest} in \eqref{rep3.1} gives
\begin{equation}
\label{rep4}
v(t,s)-w_0(k)
\ge \bk^{-6}
\iint\limits_{\lts{\kz,k}}
\frac{\uv(\tau,\xi )^2}{1+\xi+\tau}\Psi\left(\frac{\uv(\tau,\xi )}{ M(\tau,\xi )}\right)d\xi d\tau,
\end{equation}
for all $(t,s)\in\JT{\kz,k}T$.

\subsection*{Iterative sequence of lower bounds}

We now use \eqref{rep4} to construct an iterative sequence
$\{\uv_n(t,s)\}$ of lower bounds for $v(t,s)$ in $\JT{\kz,k}T$,
starting from the constant lower bound
$\uv_0(t,s)=w_0(k)$.
The argument depends on the monotonicity properties of $\Psi$, and therefore the cases 
$1<\gamma\le2$ and $\gamma>2$ will be treated separately.

First, consider the case  $1<\gamma\le2$.  
Suppose that  $\uv$ satisfies \eqref{genlb}.  Then using \eqref{MLB}, we have 
\begin{equation*}
1+\frac{\uv(t,s)}{M(t,s)}\le \frac{\uv(t,s)}{w_0(k)}+\frac{\uv(t,s)}{\um}
=\left(\frac{1+w_0(k)/\um}{w_0(k)}\right)\uv(t,s),
\end{equation*}
so that \eqref{Phi2} implies 
\begin{equation*}
\begin{aligned}
\Psi\left(\frac{\uv(t,s)}{M(t,s)}\right)
\ =\ &
\frac{\gamma-1}2  \left(1+\frac{\uv(t,s)}{M(t,s)} \right)^{\gamma-2}\\
\ \ge\ &
\frac{\gamma-1}2 \left(\frac{w_0(k)}{1+w_0(k)/\um}\right)^{2-\gamma}\uv(t,s)^{\gamma-2}.
\end{aligned}
\end{equation*}

With this lower bound, \eqref{rep4} yields
\begin{equation}
\label{rep6}
v(t,s)-w_0(k)
\ge 
2\Lambda(k) \iint\limits_{\lts{\kz,k}}
\frac{\uv(\tau,\xi )^\gamma}{1+\xi+\tau}d\xi d\tau,
\end{equation}
where
\begin{equation*}
\Lambda(k)=\frac{\gamma-1}{4\bk^6} \left(\frac{w_0(k)}{1+w_0(k)/\um}\right)^{2-\gamma}.
\end{equation*}

Let us  start with $\uv(t,s)=w_0(k)$.  
Changing variables to $a=\xi+\tau$ and $b=\xi-\tau$ in
\eqref{rep6} yields
\begin{equation}
\label{veezero}
\begin{aligned}
v(t,s)-w_0(k)
\ \ge \ &
2\Lambda(k) w_0(k)^\gamma
\iint\limits_{\lts{\kz,k}}
\frac{1}{1+\xi+\tau}d\xi d\tau\\
\ =\ &
2\Lambda(k) w_0(k)^\gamma
\int_{1}^{s+t}\int_{s-t}^k\frac1{1+a}\;\half dbda\\
\ =\ & \Lambda(k) w_0(k)^\gamma
 (k-s+t)\ln\frac{1+s+t}{2}.
\end{aligned}
\end{equation}

Motivated by \eqref{veezero}, suppose that
 we have a lower bound of the form
\begin{equation*}
v(t,s)\ge \uv(t,s)= w_0(k)+w(t,s), \quad (t,s)\in \JT{\kz,k}T,
\end{equation*}
with
\begin{equation*}
w(t,s)=C(k-s+t)^q\left(\ln\frac{1+s+t}{2}\right)^q,\quad C>0,\;q\ge0.
\end{equation*}

Again by \eqref{rep6},
since $\uv> w$, we obtain
\begin{equation}
\label{genlb1}
\begin{aligned}
v(t,s)& -w_0(k)\\
\ \ge \ &
2\Lambda(k)\iint\limits_{\lts{\kz,k}}
\frac{w(\tau,\xi)^\gamma}
{1+\xi+\tau}d\xi d\tau\\
\ = \ &
2\Lambda(k)\iint\limits_{\lts{\kz,k}}
\frac{C^\gamma(k-\xi+\tau)^{\gamma q}\left(\ln\frac{1+\xi+\tau}{2}\right)^{\gamma q}}
{1+\xi+\tau}d\xi d\tau\\
\ =\ &
2\Lambda(k) C^\gamma
\int_1^{s+t}\int_{s-t}^k\frac{(k-b)^{\gamma q}\left(\ln\frac{1+a}{2}\right)^{\gamma q}}{1+a}\;\half dbda\\
\ =\ & \Lambda(k) C^\gamma \frac1{(\gamma q+1)^2} 
(k-s+t)^{\gamma q+1}\left(\ln\frac{1+s+t}{2}\right)^{\gamma q+1}.
\end{aligned}
\end{equation}

This sets up an iterative sequence of lower bounds
\begin{subequations}
\begin{equation}
\label{seq1}
v(t,s)
\ge\uv_n(t,s)= w_0(k)+w_n(t,s),\quad n=1,2,\ldots
\end{equation}
\begin{equation}
\label{seq2}
w_n(t,s)=C_n(k-s+t)^{q_n}\left(\ln\frac{1+s+t}{2}\right)^{q_n},
\end{equation}
\end{subequations}
in which by \eqref{veezero}, \eqref{genlb1}
the sequences are determined recursively
\begin{align*}
&q_1=1,&& q_{n}=\gamma q_{n-1}+1,&& n=2,3,\ldots\\
 & C_1=\Lambda(k) w_0(k)^\gamma,&& C_{n}=\Lambda(k) C_{n-1}^\gamma(q_n)^{-2},
&& n=2,3,\ldots.
\end{align*}
The recursions yield
\begin{subequations}
\begin{align}
\label{queen}
&q_n=\frac{\gamma^n-1}{\gamma-1}, 
&& n=1,2,\ldots\\
\label{been}
&B_1=1,\quad B_n=B_{n-1}(q_n^2)^{\gamma^{-n}},&& n=2,3,\ldots,\\
\label{ceen}
&C_n=\Lambda(k)^{q_n}
\left(\frac{w_0(k)}{B_n}\right)^{\gamma^n},
&& n=1,2,\ldots.
\end{align}
\end{subequations}
We shall now show that
for any $1<\gamma\le2$, the sequence $\{B_n\}$ defined in \eqref{queen}, \eqref{been}
 is increasing and bounded above, implying that it converges to a finite limit $B_\infty(\gamma)$.

By induction, the terms of the sequence $\{B_n\}$ are given by the product
\begin{equation*}
B_n=\prod_{j=2}^n(q_j^2)^{\gamma^{-j}},\quad n=2,3,\ldots.
\end{equation*}

Since $\gamma>1$, we have $q_n>1$, 
 and so $(q_n^2)^{\gamma^{-n}}>1$, for $n=2,3,\ldots$.
 It follows that 
$\{B_n\}$ is strictly increasing.

For $1<\gamma\le 2$, we have
\[
q_n=\frac{\gamma^{n}-1}{\gamma-1}\le \frac{\gamma^{n}}{(\gamma-1)^n},
\]
and thus,
\[
B_n
\le\prod_{j=2}^n\left(\frac{\gamma}{\gamma-1}\right)^{2j\gamma^{-j}}
=\left(\frac{\gamma}{\gamma-1}\right)^{\sum_{j=2}^n2j\gamma^{-j}}
\le \left(\frac{\gamma}{\gamma-1}\right)^{\sum_{j=2}^\infty 2j\gamma^{-j}}.
\]
The infinite sum is convergent, so  $\{B_n\}$ is bounded above.

Thus, having shown that $\{B_n\}$ converges,   we find from \eqref{queen}, \eqref{ceen} that
\begin{equation*}
C_n^{1/q_n}
=\Lambda(k)\left(\frac{w_0(k)}{ B_n}\right)^{\gamma^{n}/q_n}
 \to \Lambda(k) \left(\frac{w_0(k)}{ B_\infty(\gamma)}\right)^{\gamma-1},
\end{equation*}
as $n\to\infty$.
On the other hand, since
$v(t,s)\ge w_0(k)>0$, 
we have $v(t,s)^{1/q_n}\to1$, as $n\to\infty$.
By \eqref{seq1}, \eqref{seq2}, we obtain
\begin{align*}
1
\ =\ &
\lim_n v(t,s)^{1/q_n}\\
\ \ge\ & \lim_n w_n(t,s)^{1/q_n}\\
\ =\ &\Lambda(k)\left(\frac{w_0(k)}{ B_\infty(\gamma)}\right)^{\gamma-1}
(k-s+t)\ln\frac{1+s+t}{2},
\end{align*}
for all $(t,s)\in \JT{\kz,k}T$.  In particular, if we let $(t,s)\to(T,\kz+T)$  from within $\JT{\kz,k}T$,
 this yields the upper bound
\begin{equation}
\label{lifespanub}
\begin{aligned}
\frac{1+\kz}{2}+T
\ \le\ & 
\exp \left[  \left(\frac{1}{(k-\kz)\Lambda(k)}\right) 
\left(\frac{ B_\infty(\gamma)}{w_0(k)}\right)^{\gamma-1} \right]\\
\ =\ & 
 \exp\bigg[\left(\frac{4B_\infty(\gamma)^{\gamma-1}}{\gamma-1}\right)
 \left( \frac{\bk^6}  {(k-\kz)w_0(k)}\right)\\
&\qquad\qquad\qquad\times  \left(1+\frac{ 3w_0(k)}{(1-k)^4}\right)^{2-\gamma}\bigg],
\end{aligned}
\end{equation}
for all $\kz<k<1$.

Consider now the case when $\gamma>2$. 
For any $\uv$ satisfying \eqref{genlb}, 
 we can combine \eqref{Phi2} and \eqref{rep4} to obtain
\begin{equation*}
v(t,s)-w_0(k)
\ge 
\frac{\gamma-1}{2\bk^6}\iint_{\lts{\kz,k}}
\frac{\uv(\tau,\xi )^2}{1+\xi+\tau}d\xi d\tau.
\end{equation*}
Observe that this is a special case of \eqref{rep6} if we take $\gamma=2$ and then
replace  $2\Lambda(k)$  by $(\gamma-1)/(2\bk^6)$.
Thus, from \eqref{lifespanub}, we obtain the upper bound
\begin{equation}
\label{lifespanub2}
\frac{1+\kz}{2}+T
\le\exp\left[\left(\frac{4B_\infty(2)}{\gamma-1}\right)\left( \frac{\bk^6  }{(k-\kz)w_0(k)}\right)\right].
\end{equation}
(Note that \eqref{lifespanub}, \eqref{lifespanub2} are consistent with $T\ge\bk/2=(1-\kz)/2$.)

If we now take
\begin{equation*}
\Gamma(\gamma)=\frac{4}{\gamma-1}
\begin{cases}
B_\infty(\gamma)^{\gamma-1},&1<\gamma\le2\\
B_\infty(2),&2<\gamma,
\end{cases}
\end{equation*}
then \eqref{lifespanub}, \eqref{lifespanub2} show
that 
\[
\frac{1+\kz}2+T\le T(\kz,k),\quad \kz<  k<1.
\]
Taking the infimum over $k$ gives
\[
T\le T_0(\kz),
\]
as claimed.
\qed

\section{Proof of Theorem~\ref{thm2}}

The idea is to derive an energy inequality in moving domains $\vv_\alpha$.
 We begin by combining the mass conservation law with the entropy inequality 
 to obtain a differential inequality for a nonnegative perturbed
 energy density 
 $\mathcal X$ and flux $\mathcal Y$. After estimating $\mathcal Y$
  by $\mathcal X$, we test against a distance function adapted to $\vv_\alpha$
    and derive a propagation criterion for the quantity $F(\alpha;\vv)$.

The boundedness and monotonicity of $F(\cdot;\vv)$ are immediate from
\eqref{ellinfbound} and from the monotonicity of the sets $\vv_\alpha$ in
$\alpha$. We record the proof of left-continuity.
Let 
\[
r(x)=\operatorname{dist}(x,\rr^3\setminus\vv)
\]
 and let
$0<\alpha_n\uparrow\alpha\le1$. Since $\vv$ is closed with Lipschitz
boundary,
\[
\vv_\alpha=\{(t,x)\in S_T:x\in\vv,\ t\le \alpha r(x)\}.
\]
Moreover,
\[
\bigcup_n \vv_{\alpha_n}
=
\vv_\alpha\setminus
\{(t,x)\in S_T:t=\alpha r(x)>0\}.
\]
 The exceptional set is
contained in the graph $t=\alpha r(x)$ and therefore has 
 measure zero in $\rr^4$. Hence, we have
\[
F(\alpha_n;\vv)=\|\rho^{-1}(\rho-1)\|_{L^\infty(\vv_{\alpha_n})}+\|u\|_{L^\infty(\vv_{\alpha_n})}
\uparrow
F(\alpha;\vv).
\]

\subsection*{Relative energy inequality}

Suppose that $\varphi:\st\to\rr_{\ge0}$
is a nonnegative compactly supported  Lipschitz function
with  $\supp\varphi(0,\cdot)\subset\vv$.
Then 
from  \eqref{weakformulation1} and the assumption that $F(0;\vv)=0$,
 we get
\begin{equation*}
\iint_{S_T}[(\rho-1)\partial_t\varphi+\rho u\cdot \nabla_x\varphi]dxdt
=-\int_{\vv}(\rho_0-1)\varphi(0,\cdot)dx
=0.
\end{equation*}
Since the solution is entropy-admissible on $S_T $ and $F(0;\vv)=0$,  
 we have from \eqref{admissible}
\begin{multline*}
\iint_{S_T}\left[\left(\half\rho|u|^2+\tfrac{1}{\gamma(\gamma-1)}\left(\rho^{\gamma}-1
\right)\right)\right.\partial_t\varphi \\
+\left.\left(\half\rho|u|^2+\tfrac{1}{\gamma-1}\rho^{\gamma}\right)u\cdot\nabla_x\varphi\right]dxdt
\ge 0.
\end{multline*}

Combining these relations yields
\begin{equation}
\label{wkentropy}
\iint_{S_T}(\mathcal X(\rho,u)\partial_t\varphi+\mathcal Y(\rho,u)\cdot\nabla_x\varphi ) dxdt\ge0,
\end{equation}
in which
\begin{equation*}
\begin{aligned}
\mathcal X(\rho,u)
\ =\ &\half\rho|u|^2+\tfrac{1}{\gamma-1}\Phi(\rho-1),\\
\mathcal Y(\rho,u)\ =\ & \left(\half\rho|u|^2
+\tfrac{1}{\gamma-1}\left(\rho^{\gamma}-\rho\right)
\right)u.
\end{aligned}
\end{equation*}

Next we estimate $\mathcal Y(\rho,u)$, temporarily treating $(\rho,u)$
as independent variables.  Using \eqref{Phi1}, we can write
\begin{equation*}
\mathcal Y(\rho,u)=(\rho-1)u+\left(\half\rho|u|^2+\tfrac{\gamma}{\gamma-1}\Phi(\rho-1)\right)u,
\end{equation*}
so that
\begin{align*}
|\mathcal Y(\rho,u)|
\ \le\ & 
 \half\rho|u|^2+\half\rho^{-1}(\rho-1)^2+\gamma\mathcal X(\rho,u)|u|\\
\ =\ &
\mathcal X(\rho,u)+\mathcal Z(\rho-1)+\gamma\mathcal X(\rho,u)|u|,
\end{align*}
with
\begin{equation*}
\mathcal Z(\rho-1)= \half\rho^{-1}(\rho-1)^2-\tfrac1{\gamma-1}\Phi(\rho-1).
\end{equation*}

Regarded as a function of $\eta=\rho-1$, $\mathcal Z$ satisfies
$\mathcal Z(0)=\mathcal Z'(0)=\mathcal Z''(0)=0$.
Moreover, $\Phi(0)=\Phi'(0)=0$ and
$\Phi''(\eta)>0$ for $\eta>-1$.
To compare $\mathcal Z$ with $\Phi$, we define
the function\footnote{When $\gamma=2$, we have $Q(\rho)=-1$.}
\begin{equation*}
Q(\rho)=
\begin{cases}
\displaystyle
\frac{\rho\mathcal Z(\rho-1)}{(\rho-1)\Phi(\rho-1)},&\rho>0,\;\rho\ne1\\
\mathcal Z'''(0)/(3\Phi''(0)),&\rho=1.
\end{cases}
\end{equation*}

The assigned value makes $Q$ continuous on $\rr_{>0}$.
Moreover, $|Q(\rho)|$ has finite limits as $\rho\to0$ and
$\rho\to\infty$.
Thus,
there is a constant $C(\gamma)>0$ such that
\begin{equation*}
|Q(\rho)|\le C(\gamma),\quad \rho\in\rr_{>0}.
\end{equation*}
It follows that $|\mathcal Z(\rho-1)|\le C(\gamma)\rho^{-1}|\rho-1|\Phi(\rho-1)$, and so
\begin{equation}
\label{YYest}
|\mathcal Y(\rho,u)| \le
\mathcal X(\rho,u)+C(\gamma)(\rho^{-1}|\rho-1|+|u|)\mathcal X(\rho,u),
\end{equation}
for a possibly larger constant $C(\gamma)$.

From  \eqref{YYest}, we deduce
\begin{multline*}
\iint_{S_T}|\mathcal Y(\rho,u)\cdot\nabla\varphi|dxdt
\le \iint_{S_T}\mathcal X(\rho,u)|\nabla\varphi|dxdt\\
+ C(\gamma)\iint_{S_T}(\rho^{-1}|\rho-1|+|u|)\mathcal X(\rho,u)|\nabla\varphi|dxdt.
\end{multline*}
Combining this estimate with \eqref{wkentropy} yields the relative energy inequality
\begin{multline}
\label{enineq1}
-\iint\limits_{S_T}\mathcal X(\rho,u)(\partial_t\varphi+|\nabla_x\varphi|)dxdt\\
\le C(\gamma)\iint\limits_{S_T}(\rho^{-1}|\rho-1|+|u|)\mathcal X(\rho,u)|\nabla\varphi|dxdt.
\end{multline}

\subsection*{Estimate on moving domains}

We now construct a family of test functions to use in \eqref{enineq1}.
Let $\vv$, $d_\alpha$,  and $\vv_\alpha$ be defined as in the statement of the theorem.
By construction, the function $d_\alpha$ is Lipschitz continuous and satisfies the eikonal equation
\begin{equation}
\label{eikonal}
\partial_td_\alpha+|\nabla_x d_\alpha|/\alpha=0,\isp{a.e.} \rr_{\ge0}\times\rr^3.
\end{equation}
To obtain a compactly supported test function, we introduce the time cut-off  $h$ from \eqref{hdef}
and define
\begin{equation*}
\varphi_{\alpha}(t,x)=h\left(\frac{\tau-t}\eps\right)d_\alpha(t,x),\quad 0<\tau<T.
\end{equation*}
Since $\vv$ is bounded,
this is a nonnegative compactly supported  Lipschitz test function on $[0,T)\times\rr^3$ with
\begin{equation*}
\supp\varphi_{\alpha}
\subset\vv_\alpha\cap \overline{S_\tau}.
\end{equation*}

Since $h$ is nonnegative and nondecreasing, and since \eqref{eikonal} holds, we have
\[
|\nabla_x\varphi_\alpha|
=
h\left(\frac{\tau-t}{\eps}\right)
\]
and
\[
-\big(\partial_t\varphi_\alpha+|\nabla\varphi_\alpha|\big)
\ge
h\left(\frac{\tau-t}{\eps}\right)\left(\frac1\alpha-1\right),
\]
a.e. on $\supp\varphi_\alpha$.

Substituting the test function $\varphi_{\alpha}$ in \eqref{enineq1}, we obtain 
\begin{equation*}
\begin{aligned}
\iint\limits_{\supp\varphi_\alpha}&
h\left(\frac{\tau-t}\eps\right)\left(\frac1\alpha-1\right)\mathcal X(\rho,u)dxdt\\
\ \le\ & 
C(\gamma)(\|\rho^{-1}(\rho-1)\|_{L^\infty(\vv_\alpha)}
+\|u\|_{L^\infty(\vv_\alpha)})\\
&\times
\iint\limits_{\supp\varphi_\alpha}
h\left(\frac{\tau-t}\eps\right)\mathcal X(\rho,u)dxdt\\
\ =\ & C(\gamma)F(\alpha;\vv)
\iint\limits_{\supp\varphi_\alpha}
h\left(\frac{\tau-t}\eps\right)\mathcal X(\rho,u)dxdt.
\end{aligned}
\end{equation*}
Letting $\eps\to0$ and then $\tau\to T$   gives
\begin{equation}
\label{enineq2}
\left(\frac1\alpha-1\right)\iint_{\vv_\alpha}\mathcal X(\rho,u)dxdt
\le 
C(\gamma)F(\alpha;\vv)\iint_{\vv_\alpha}\mathcal X(\rho,u)dxdt.
\end{equation}
Define the function
\begin{equation*}
g(\alpha;\vv)=\iint_{\vv_\alpha}
\mathcal X(\rho,u)dxdt,\quad 0<\alpha\le1.
\end{equation*}
Note that $g$ is bounded and  nondecreasing.
For $t>0$, the set $\vv_\alpha$ is given by
\[
\vv_\alpha\cap((0,T)\times\rr^3)
=
\{(t,x)\in (0,T)\times\rr^3:
t\le \alpha\,\dist(x,\rr^3\setminus\vv)\}.
\]
Thus, the characteristic functions of $\vv_\alpha$ vary pointwise a.e.\
continuously with respect to $\alpha$. By dominated convergence, $g$ is
continuous.
Moreover, since $\mathcal X(\rho,u)=0$ if and only if $(\rho,u)=(1,0)$,
we have $g(\alpha;\vv)=0$ if and only if $F(\alpha;\vv)=0$.

\subsection*{The critical parameter $\alpha(V)$}

We now identify the largest value of $\alpha$ for which the disturbance must vanish.
Equation \eqref{enineq2} tells us that 
\begin{equation}
\label{point}
g(\alpha;\vv)\le\frac{C(\gamma)\alpha}{1-\alpha}
 F(\alpha;\vv)g(\alpha;\vv),\quad 0<\alpha<1.
\end{equation}
Consider the  set
\begin{equation*}
\mathcal A(\vv)=\left\{0<\alpha<1: \frac{C(\gamma)\alpha}{1-\alpha}
F(\alpha;\vv)<1\right\}.
\end{equation*}
If $\alpha\in\mathcal A(\vv)$, then \eqref{point} implies that  $g(\alpha;\vv)=0$.
Consequently,
\begin{align*}
\mathcal A(\vv)\ \subset\ & \{0<\alpha<1: g(\alpha;\vv)=0\}\\
\ =\ &\{0<\alpha<1: F(\alpha;\vv)=0\}\\
\ \subset\ &
\mathcal A(\vv).
\end{align*}
Since both $F(\alpha;\vv)$ and $\alpha/(1-\alpha)$ are nondecreasing,
$\mathcal A(\vv)$ is an interval.
From \eqref{ellinfbound}, $F$ is bounded.  Hence, $\mathcal A(\vv)\ne\emptyset$.
Let
\begin{equation*}
\alpha(\vv)=\sup\mathcal A(\vv).
\end{equation*}
Since
\[
\mathcal A(\vv)=\{0<\alpha<1:F(\alpha;\vv)=0\},
\]
this agrees with the definition of $\alpha(\vv)$ in the statement of the theorem.
The continuity of $g$ then implies that
\[
F(\alpha;\vv)=g(\alpha;\vv)=0,\qquad 0<\alpha<\alpha(\vv).
\]
If $\alpha(\vv)<1$, the same holds at $\alpha=\alpha(\vv)$, and
$\mathcal A(\vv)=(0,\alpha(\vv)]$. If $\alpha(\vv)=1$, left-continuity gives
$F(1;\vv)=0$.

If $\alpha(\vv)<1$, then
\begin{equation*}
\frac{C(\gamma)\alpha(\vv)}{1-\alpha(\vv)}
F\left(\alpha(\vv)^+;\vv\right)\ge1.
\end{equation*}
From this,  we obtain the  lower bound 
\begin{equation*}
F(\alpha;\vv)\ge F\left(\alpha(\vv)^+;\vv\right)\ge
\frac{1-\alpha(\vv)}{C(\gamma)\alpha(\vv)},\quad \alpha(\vv)<\alpha\le1.
\end{equation*}

 By \eqref{ellinfbound}, we have $F(\alpha;\vv)\le m$.
So if $\alpha<(1+C(\gamma)m)^{-1}$, then
\begin{equation*}
 \frac{C(\gamma)\alpha}{1-\alpha}F(\alpha;\vv)
 < \frac{C(\gamma)m\alpha}{1-\alpha}<1.
\end{equation*}
Thus,
\begin{equation*}
1\ge \alpha(\vv)\ge (1+C(\gamma)m)^{-1}.
\end{equation*}

Suppose that the solution has a continuous representative on
$\vv_{\hat\alpha}$, for some $0<\hat\alpha\le1$.  
If  $\vv_{\hat\alpha}$  is compact, then  the representative 
is uniformly continuous on  $\vv_{\hat\alpha}$, and so,
$F(\alpha;\vv)$ is continuous on $(0,\hat\alpha]$.  
Having established that $F(\alpha;\vv)$ has a jump discontinuity at
$\alpha(\vv)$ whenever $\alpha(\vv)<1$, we conclude that
$\hat\alpha\le\alpha(\vv)$.  Consequently, $F(\alpha;\vv)$ vanishes on
$(0,\hat\alpha]$ and  the
 disturbance vanishes a.e.\  on $\vv_{\hat\alpha}$.
 
 If $\vv_{\hat\alpha}$ is not compact, then
 \[
 \vv_{\hat\alpha}\cap \bigl([0,T']\times\rr^3\bigr)
 \]
  is compact, for every
  $0<T'<T$.   Apply the preceding
 argument to the restriction of the solution to $\st[T']$ to obtain
 that the  disturbance vanishes a.e.\  on $\vv_{\hat\alpha}\cap\st[T']$ and then
 let $T'\uparrow T$.

\qed

\section{Proof of  Corollary~\ref{cor1}}

To prove the first assertion, fix $N\in\mathbb N$, $N>R$, and let
\[
\vv=\{x\in\rr^3:R\le |x|\le N\}.
\]
Then $F(0;\vv)=0$ by \eqref{esupp0.0}. Hence Theorem~\ref{thm2} implies that
$F(\alpha;\vv)=0$ for every $0\le \alpha\le \alpha(\vv)$, where
\[
\alpha(\vv)\ge (1+C(\gamma)m)^{-1}=c_0^{-1}.
\]
Taking $\alpha=c_0^{-1}$, we obtain
\[
\rho=1,\qquad u=0
\]
a.e. in
\[
\vv_{c_0^{-1}}
=
\{(t,x)\in S_T:R+c_0t\le |x|\le N-c_0t\}.
\]
Since $N$ is arbitrary, letting $N\to\infty$ gives
\[
\esupp(\rho-1,u)
\subset
\{(t,x)\in S_T:|x|\le R+c_0t\},
\]
as claimed.

For the second assertion, the discussion in   Section~\ref{notation} implies
that the restriction of $(\rho,u)$ to $[\sigma,T)\times\rr^3$ is an
entropy-admissible weak solution with initial data
$(\rho_\sigma,u_\sigma)$. Applying the first assertion to the
time-translated solution gives the desired inclusion.
\qed

\section{Proof of Theorem~\ref{thm3}}

\subsection*{Transition time and support properties}

For brevity, write
\begin{equation*}
\sss= \esupp(\rho-1,u).
\end{equation*}
For $0<t<T$, define the nested family
\[
\mathcal E_t
=
\left\{
(s,x)\in S_t:|x|>1+s
\right\},
\]
and set $\mathcal E_0=\emptyset$.
Thus
\[
\mathcal T
=
\left\{
0<t<T:\sss\cap\mathcal E_t\ne\emptyset
\right\}.
\]
Under the assumption \eqref{wkdatacond},  Theorem~\ref{thm1} implies that 
$(T_0(\kz),T)\subset\mathcal T$.  Hence $\tau=\inf\mathcal T$ is well-defined and satisfies
$0\le\tau\le T_0(\kz)$.  
By the definition of $\tau$ and the nestedness of the sets $\mathcal E_t$, \eqref{nospread} holds.

We proceed to verify \eqref{exteriorsupp} and \eqref{spread0}.
If $\tau=0$, then Corollary~\ref{cor1} and \eqref{esupp0.0} give
\[
\esupp(\rho-1,u)\subset\{(t,x)\in\st:|x|\le1+c_0t\}.
\]
Next, suppose that $\tau>0$, and let  $\{\sigma_j\} $ be a sequence of Lebesgue times
with $0<\sigma_j\uparrow \tau$.
 It is straightforward to verify from \eqref{nospread} that the 
corresponding Lebesgue values satisfy
\[
\esupp(\rho_{\sigma_j}-1,u_{\sigma_j})\subset\{x\in\rr^3:|x|\le 1+\sigma_j\}.
\]
Using  Corollary~\ref{cor1} again, we obtain 
\[
\sss\cap\bigl((\sigma_j,T)\times\rr^3\bigr) \subset 
\{(t,x)\in (\sigma_j,T)\times\rr^3:\ |x|\le 1+\sigma_j+c_0(t-\sigma_j)\}.
\]
Letting $\sigma_j\uparrow\tau$, we conclude that
\[
\sss\cap\left((\tau,T)\times\rr^3\right) \subset 
\{(t,x)\in (\tau,T)\times\rr^3:\ |x|\le 1+\tau+c_0(t-\tau)\}.
\]
This proves \eqref{exteriorsupp}.

Since $\mathcal T$ is an upper interval with infimum $\tau$,
\[
\mathcal S\cap\mathcal E_t\ne\emptyset,
\qquad \tau<t<T.
\]
As $\mathcal E_t$ is open in $\st$, the definition of essential
support implies that $\mathcal S\cap\mathcal E_t$ has positive measure.

Suppose first that $\tau>0$. Since
\[
\mathcal E_\sigma\cap\sss=\emptyset,
\qquad 0\le\sigma<\tau,
\]
and $\mathcal E_\sigma\uparrow\mathcal E_\tau$ as
$\sigma\uparrow\tau$, we have
$\mathcal E_\tau\cap\sss=\emptyset$. The same conclusion holds when
$\tau=0$, since $\mathcal E_0=\emptyset$. Since
$\{\tau\}\times\rr^3$ has measure zero,
\[
\sss\cap\mathcal E_t\isp{and}\sss\cap\mathcal E_t\cap\bigl((\tau,t)\times\rr^3\bigr)
\]
have the same positive measure. Combined with
\eqref{exteriorsupp}, this proves \eqref{spread0}.

\subsection*{Shifted version of Theorem~\ref{thm2}}
It remains to prove the assertion concerning the jump profile.
For this purpose, we first observe that the argument of
Theorem~\ref{thm2} may be shifted to the time $\tau$ without
assuming that $\tau$ is a Lebesgue time.
Fix $ \tau'$ such that $\tau<\tau'<T$.  We shall work on the set $\st[\tau']$.
Let $\vv\subset\rr^3$ be a nonempty, closed,
and bounded set with Lipschitz boundary such that
\begin{subequations}
\begin{equation}
\label{obs1}
(\rho_0-1,u_0)=0\quad\text{a.e. on }\vv
\end{equation}
and
\begin{equation}
\label{obs2}
(\rho-1,u)=0\quad\text{a.e. on }(0,\tau)\times \vv.
\end{equation}
\end{subequations}
For $0<\alpha\le1$, set
\[
d_{\alpha}(t,x)
=
\left(
\dist(x,\rr^3\setminus \vv)
-\frac{(t-\tau)_+}{\alpha}
\right)_+.
\]
Given $\tau<t^\ast<\tau'$ and $0<\eps<t^\ast-\tau$, use in the
relative-energy inequality \eqref{enineq1} the test function
\[
\varphi_{\alpha}(t,x)
=
h\left(\frac{t^\ast-t}{\eps}\right)d_{\alpha}(t,x),
\]
where $h$ is the time cutoff used in the proof of
Theorem~\ref{thm2}. Then
\[
\varphi_{\alpha}\in C^{0,1}_0(S_T)
\]
and
\[
\varphi_{\alpha}=d_{\alpha}
\qquad\text{on }[0,t^\ast-\eps]\times\rr^3.
\]
Its support in $(0,\tau)\times\rr^3$ is contained in
$(0,\tau)\times \vv$, where the disturbance vanishes. 
The condition needed to apply \eqref{enineq1} is satisfied, since
$\supp\varphi_\alpha(0,\cdot)\subset \vv$ and the 
 initial disturbance vanishes a.e.\ on $\vv$.
Thus all contributions
from $t<\tau$ vanish, while on $(\tau,t^\ast)\times\rr^3$ the argument
in the proof of Theorem~\ref{thm2} applies after the translation
$t\mapsto t-\tau$. Letting $\eps\downarrow0$ and then
$t^\ast\uparrow \tau'$ gives the corresponding conclusions for the
shifted family.

More precisely, let
\[
\vv_{\alpha}
=
\supp d_{\alpha}
\cap\bigl((\tau,\tau')\times\rr^3\bigr)
\]
and set
\[
F(\alpha;\vv)
=
\|\rho^{-1}(\rho-1)\|_{L^\infty(\vv_{\alpha})}
+
\|u\|_{L^\infty(\vv_{\alpha})},
\qquad 0<\alpha\le1,
\]
with $F(0;\vv)=0$. Then $F(\cdot;\vv)$ is bounded,
nondecreasing, and left-continuous. Moreover,
\[
\alpha(\vv)
=
\sup\left\{
0\le\alpha\le1:F(\alpha;\vv)=0
\right\}
\ge c_0^{-1}.
\]
If $\alpha(\vv)<1$, then
\[
F(\alpha;\vv)=0,
\qquad
0\le\alpha\le\alpha(\vv),
\]
whereas
\[
F(\alpha;\vv)
\ge
F\bigl(\alpha(\vv)^+;\vv\bigr)>0,
\qquad
\alpha(\vv)<\alpha\le1.
\]

\subsection*{Construction of the jump profile}

By \eqref{spread0}, we may choose
\[
(t_0,x_0)\in \sss\cap\left\{(t,x)\in(\tau,\tau')\times\rr^3:
1+t< |x|\le1+\tau+c_0(t-\tau)\right\},
\]
and then set $\nu=x_0/|x_0|$.
 Let
\[
R>1+\tau+2c_0(\tau'-\tau)
\]
so that
\begin{equation}
\label{Rprop}
1+\tau + c_0(t-\tau)\le R-c(t-\tau),\quad 1\le c\le c_0,\quad \tau<t<\tau',
\end{equation}
and define
\[
\vv=\{x\in\rr^3:\nu\cdot x\ge1+\tau,\ |x|\le R\}.
\]
If $x\in\vv$, then $|x|\ge \nu\cdot x\ge 1+\tau\ge1$.  Thus, by \eqref{esupp0.0}, the condition
\eqref{obs1} holds.
At the same time, if $(t,x)\in(0,\tau)\times\vv$, then $|x|> 1+t$, and we see from \eqref{nospread}
that \eqref{obs2} is also valid.

Following the notation above, for   $ 0<\alpha\le1$, we have
\begin{multline*}
 \vv_{\alpha}=\{(t,x)\in (\tau,\tau')\times\rr^3: \\
 \nu\cdot x\ge 1+\tau+(t-\tau)/\alpha,\; |x|\le R-(t-\tau)/\alpha\}.
\end{multline*}
Letting $c=\alpha^{-1}$, we can write
\begin{multline*}
\vv_{c^{-1}}=\{(t,x)\in (\tau,\tau')\times\rr^3: \\
\nu\cdot x\ge 1+\tau+c(t-\tau),\;  |x|\le R-c(t-\tau)\},
\end{multline*}
for $1\le\ c<\infty$.  

From \eqref{exteriorsupp} and \eqref{Rprop}, for $1\le c\le c_0$ the radial
cutoff is inactive on $\sss$, and hence
\[
\vv_{c^{-1}}\cap\sss=\mathcal P_c\cap\sss,\quad 1\le c\le c_0,
\]
with
\begin{multline*}
\mathcal P_c=\{(t,x)\in (\tau,\tau')\times\rr^3: \\
\nu\cdot x\ge 1+\tau+c(t-\tau),\; |x|\le 1+\tau+c_0(t-\tau)\}.
\end{multline*}

On the other hand, using the notation \eqref{jumpplane}, we also have
\begin{equation*}
{\mathcal H}_{\tau,\tau'}^{\nu,c}\cap\sss 
= \mathcal P_c\cap\sss.
\end{equation*}
\begin{figure}[H]
\caption{Schematic geometry in the $(\nu\cdot x,t)$-plane.
The line $\nu\cdot x=1+t$ represents the supporting hyperplane,
in the direction $\nu$, to the classical spherical front
$|x|=1+t$, while
$\nu\cdot x=1+\tau+c_0(t-\tau)$ is the corresponding supporting
hyperplane for the outer propagation bound
$|x|=1+\tau+c_0(t-\tau)$.
The shaded region to the right of
$\nu\cdot x=1+\tau+c(t-\tau)$, for $\tau<t<\tau'$, is
$\mathcal H_{\tau,\tau'}^{\nu,c}$, where $1<c<c_0$.
This region is unbounded in the $\nu\cdot x$-direction.}
\label{diagram3}

\centering
\vspace{0.3cm}
\begin{tikzpicture}[x=0.78cm,y=0.78cm]
  \def\xP{1.55}
  \def\ytau{1.30}
  \def\ytaup{3.15}
  \def\yT{5.75}
  \def\xright{11.20}
  \def\xaxisend{11.85}

  \def\mone{0.72}
  \def\mc{1.42}
  \def\mczero{2.02}

  \pgfmathsetmacro{\ymidlabel}{0.5*(\ytaup+\yT)}
  \pgfmathsetmacro{\poslabel}{\ymidlabel/\yT}

  \coordinate (P) at (\xP,\ytau);

  \coordinate (A) at ({\xP+\mone*(\yT-\ytau)},\yT);
  \coordinate (B) at ({\xP+\mc*(\yT-\ytau)},\yT);
  \coordinate (C) at ({\xP+\mczero*(\yT-\ytau)},\yT);

  \coordinate (A0) at ({\xP-\mone*\ytau},0);

  \coordinate (Q) at ({\xP+\mc*(\ytaup-\ytau)},\ytaup);

  \draw[->] (0,0) -- (0,6.55) node[above] {$t$};
  \draw[->] (0,0) -- (\xaxisend,0) node[right] {$\nu\cdot x$};

  \draw[densely dashed] (0,\ytau) -- (P);
  \draw[densely dashed] (0,\ytaup) -- (Q);
  \draw[densely dashed] (0,\yT) -- (A);

  \node[left] at (0,\ytau) {\small{$t=\tau$}};
  \node[left] at (0,\ytaup) {\small{$t=\tau'$}};
  \node[left] at (0,\yT) {\small{$t=T$}};

  \fill[gray!15]
    (P) -- (Q) -- (\xright,\ytaup) -- (\xright,\ytau) -- cycle;

  \draw[gray] (P) -- (\xright,\ytau);
  \draw[gray] (Q) -- (\xright,\ytaup);
  \draw[->,gray] (8.95,2.20) -- (10.90,2.20);

  \node at (7.55,2.30) {$\mathcal H^{\nu,c}_{\tau,\tau'}$};

  \draw (A0) -- (A)
    node[pos=.74,sloped,above=4pt,fill=white,inner sep=1pt]
    {\scriptsize $\nu\cdot x=1+t$};

  \draw (P) -- (B)
    node[pos=.68,sloped,above=3pt,fill=white,inner sep=1pt]
    {\scriptsize $\nu\cdot x=1+\tau+c(t-\tau)$};

  \draw (P) -- (C)
    node[pos=.73,sloped,above=3pt,fill=white,inner sep=1pt]
    {\scriptsize $\nu\cdot x=1+\tau+c_0(t-\tau)$};
\end{tikzpicture}
\end{figure}
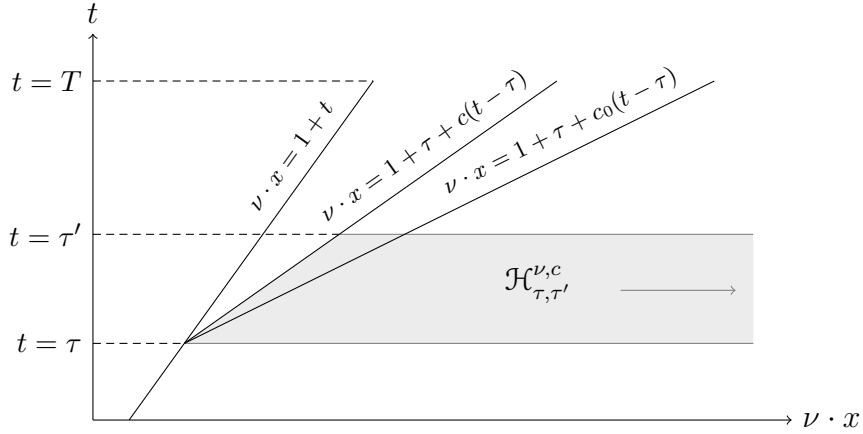
It follows that the function
\begin{equation*}
f(c)=\|\rho^{-1}(\rho-1)\|_{L^\infty\left({\mathcal H}_{\tau,\tau'}^{\nu,c}\right)}
+
\|u\|_{L^\infty\left({\mathcal H}_{\tau,\tau'}^{\nu,c}\right)}
\end{equation*}
satisfies
$F(c^{-1};\vv)=f(c)$, $1\le c\le c_0$.
If $c>c_0$, then  
\[
\vv_{c^{-1}}\cap\sss={\mathcal H}_{\tau,\tau'}^{\nu,c}\cap\sss=\emptyset.
\]
 Thus, $F(c^{-1};\vv)=f(c)=0$, for $c>c_0$.  Altogether, we have shown
 that
 \[
F(c^{-1};\vv)=f(c),\quad 1\le c<\infty.
\]

By the choice of $\nu$, we have
\[
\nu\cdot x_0=|x_0|>1+t_0.
\]
Hence, $(t_0,x_0)$ lies in the interior of
$\mathcal H_{\tau,\tau'}^{\nu,1}$, and its membership in $\mathcal S$
gives
\[
f(1)=F(1;\mathcal V)>0.
\]
Therefore, using the shifted version of Theorem~\ref{thm2} proved above, we obtain
\[
\alpha(\vv)<1.
\]
Set
\[
c_1=\alpha(\vv)^{-1}.
\]
Since $\alpha(\vv)\ge c_0^{-1}$, we have
\[
1<c_1\le c_0.
\]
Moreover,
\[
f(c)=0,
\qquad c_1\le c<\infty,
\]
whereas
\[
f(c)
\ge
F\bigl(\alpha(\vv)^+;\vv\bigr)>0,
\qquad 1\le c<c_1.
\]
Consequently,
\[
f(c_1^-)>f(c_1^+)=0.
\]

Finally, set
\[
g(c)=\|\rho-1\|_{L^\infty\left({\mathcal H}_{\tau,\tau'}^{\nu,c}\right)}+\|u\|_{L^\infty\left({\mathcal H}_{\tau,\tau'}^{\nu,c}\right)}.
\]
The function $g$ is nonincreasing because the sets
${\mathcal H}_{\tau,\tau'}^{\nu,c}$ are decreasing in $c$.
Since $\rho^{-1}$ and $\rho$ belong to $L^\infty(\st)$, we have 
\begin{multline*}
\|\rho^{-1}\|_{L^\infty(\st)}^{-1}\|\rho^{-1}(\rho-1)\|_{L^\infty\left({\mathcal H}_{\tau,\tau'}^{\nu,c}\right)}
\\ \le\|\rho-1\|_{L^\infty\left({\mathcal H}_{\tau,\tau'}^{\nu,c}\right)}
\le \|\rho\|_{L^\infty(\st)}\|\rho^{-1}(\rho-1)\|_{L^\infty\left({\mathcal H}_{\tau,\tau'}^{\nu,c}\right)}.
\end{multline*}
The preceding bounds show that $f$ and $g$ are comparable, with
constants independent of $c$.
 Hence, $g(c)=0$ for
$c_1\le c<\infty$, while $g(c)$ is bounded below by a positive constant for
$1\le c<c_1$. Therefore $g$ has a jump discontinuity at $c=c_1$, with
\[
g(c_1^-)>g(c_1^+)=0.
\]

\qed

\bibliography{CEE}
\bibliographystyle{amsplain}

\end{document}